\documentclass[a4paper, reqno]{amsart}

\usepackage{amsthm,mathrsfs,amsmath,amscd,enumerate,enumitem, amscd, xcolor, url}

\usepackage[colorlinks]{hyperref}
\usepackage{amsfonts}
\usepackage{amssymb, mathtools}
\usepackage[utf8]{inputenc}
\usepackage[defaultlines=2,all]{nowidow}

\newtheorem{thm}{Theorem}[section]

\newtheorem{lem}[thm]{Lemma}

\theoremstyle{definition}

\theoremstyle{remark}
\newtheorem{rem}[thm]{Remark}

\DeclareMathOperator{\supp}{supp}
\DeclareMathOperator{\dist}{dist}

\newcommand{\loc}{\textup{loc}}
\newcommand{\glob}{\text{glob}}
\DeclareMathOperator{\Real}{Re}

\newcommand{\BMO}{\textup{BMO}}

\setlist[enumerate,1]{label=(\roman*)}

\usepackage{marginnote}
\newcounter{BDQSc}

\numberwithin{equation}{section}
\allowdisplaybreaks

\begin{document}
\title[Hardy spaces and Laguerre expansions]
 {Maximal function characterization of Hardy spaces related to Laguerre polynomial expansions}

\author[J. J. Betancor]{Jorge J. Betancor}
    \address{Jorge J. Betancor\newline
        Departamento de An\'alisis Matem\'atico, Universidad de La Laguna,\newline
        Campus de Anchieta, Avda. Astrof\'isico S\'anchez, s/n,\newline
        38721 La Laguna (Sta. Cruz de Tenerife), Spain}
    \email{jbetanco@ull.es}

    \author[E. Dalmasso]{Estefan\'ia Dalmasso}
    \address{Estefan\'ia Dalmasso, Pablo Quijano\newline
        Instituto de Matem\'atica Aplicada del Litoral, UNL, CONICET, FIQ.\newline Colectora Ruta Nac. Nº 168, Paraje El Pozo,\newline S3007ABA, Santa Fe, Argentina}
    \email{edalmasso@santafe-conicet.gov.ar, pabloquijanoar@gmail.com}

    \author[P. Quijano]{Pablo Quijano}

    \author[R. Scotto]{Roberto Scotto}
    \address{Roberto Scotto\newline
        Universidad Nacional del Litoral, FIQ.\newline Santiago del Estero 2829,\newline S3000AOM, Santa Fe, Argentina}
    \email{roberto.scotto@gmail.com}

\thanks{The first author is partially supported by PID2019-106093GB-I00 (Ministerio de Ciencia e Innovaci\'on, Spain). The second and fourth authors are partially supported by grants PICT-2019-2019-00389 (ANPCyT), PIP-11220200101916CO (CONICET) and CAI+D 2019-015 (UNL)}
\date{\today}
\subjclass{42A50, 42B20, 42B25.}

\keywords{Hardy spaces, atoms, maximal functions, Laguerre polynomials}

\maketitle

\begin{abstract}
In this paper we introduce the atomic Hardy space $\mathcal{H}^1((0,\infty),\gamma_\alpha)$ associated with the non-doubling probability measure $d\gamma_\alpha(x)=\frac{2x^{2\alpha+1}}{\Gamma(\alpha+1)}e^{-x^2}dx$ on $(0,\infty)$, for ${\alpha>-\frac12}$. 
We obtain characterizations of  $\mathcal{H}^1((0,\infty),\gamma_\alpha)$ by using two local maximal functions. We also prove that the truncated maximal function defined through the heat semigroup generated by the Laguerre differential operator is bounded from $\mathcal{H}^1((0,\infty),\gamma_\alpha)$ into $L^1((0,\infty),\gamma_\alpha)$.
\end{abstract}

\setcounter{secnumdepth}{3}
\setcounter{tocdepth}{3}

\section{Introduction and main results}

We consider, for $\alpha>-\frac12 $, the probability measure $d\gamma_\alpha(x)=\frac{2x^{2\alpha+1}}{\Gamma(\alpha+1)}e^{-x^2}dx$ on~$(0,\infty)$. The measure $\gamma_\alpha$ is not doubling (not even locally doubling) with respect to the usual metric in $(0,\infty)$. We could define Hardy spaces as in \cite{To}, since it is clear that $\sup_{r,x\in (0,\infty)}\gamma_\alpha(I(x,r))/r<\infty$, where for every $r,x\in (0,\infty)$, we denote $I(x,r):=(x-r,x+r)\cap (0,\infty)$. However, Tolsa's definition is not satisfactory for our purposes because the harmonic analysis operators associated with Laguerre polynomial expansions are not Calder\'on-Zygmund operators in this setting. We will define the Hardy space related to the measure $\gamma_\alpha$ in $(0,\infty)$ following the ideas developed in~\cite{MM} (see also \cite{CMM2}) for the Ornstein-Uhlenbeck setting and the non-standard Gaussian measure~$d\gamma(x)=\pi^{-n/2}e^{-|x|^2}dx$ in $\mathbb{R}^n$.

In order to select a family of intervals over which $\gamma_\alpha$ is indeed doubling, we consider the admissibility function $m(x)=\min\left\{1,\frac{1}{x}\right\}$, for $x\in (0,\infty)$. Given~$a>0$, we say that an interval $I(x,r)$ with $0<r\le x$
is \textbf{$a$-admissible} if $r\leq am(x)$. The class of such intervals will be denoted by $\mathcal{B}_a$.
To simplify notation, we shall write~$\mathcal{B}:=\mathcal{B}_1$. For every $a>0$, it is easy to see that the measure $\gamma_\alpha$ is doubling on $\mathcal{B}_a$, that is, there exists $C_{\alpha,a}>0$ such that
\[\gamma_\alpha(I(x,2r))\le C_{\alpha,a}\gamma_\alpha(I(x,r)),\quad I(x,r)\in \mathcal{B}_a.\]

The atoms will be defined over these families of intervals as follows. Given $1<q\le \infty$ and $a>0$, a measurable function $b$ defined on $(0,\infty)$ is said to be an~\textbf{$(a,q,\alpha)$--atom} when $b(x)=1$ for every $x\in (0,\infty)$, or there exist $0<r\le x$ such that $I(x,r)\in \mathcal{B}_a$ and satisfying that
\begin{enumerate}
    \item $\supp(b)\subset I(x,r)$;
    \item $\|b\|_{L^q((0,\infty),\gamma_\alpha)}\le \gamma_\alpha(I(x,r))^{\frac{1}{q}-1}$, where $\frac{1}{q}=0$ when $q=\infty$;
    \item $\int_0^\infty b(y)\ d\gamma_\alpha(y)=0$.
\end{enumerate}

For $1<q\leq \infty$ and $a>0$, the atomic Hardy space $\mathcal{H}^{1,q}_a((0,\infty),\gamma_\alpha)$ consists of all of those $f\in L^1((0,\infty),\gamma_\alpha)$ such that $f=\sum_{j=0}^\infty\lambda_jb_j$ where, for every ${j\in \mathbb{N}}$,~$b_j$ is an $(a,q,\alpha)$--atom and $\lambda_j\in \mathbb{C}$ with $\sum_{j=0}^\infty|\lambda_j|<\infty$. Note that, under the above conditions, the series $\sum_{j=0}^\infty\lambda_jb_j$ converges in $L^1((0,\infty),\gamma_\alpha)$. We define, for every~$f\in \mathcal{H}^{1,q}_a((0,\infty),\gamma_\alpha)$,
\[\|f\|_{\mathcal{H}^{1,q}_a((0,\infty),\gamma_\alpha)}=\inf\sum_{j=0}^\infty|\lambda_j|,\]
where the infimum is taken over all the sequences $\{\lambda_j\}_{j\in \mathbb{N}}\in \mathbb{C}$ with $\sum_{j=0}^\infty|\lambda_j|<\infty$ and $f =\sum_{j=0}^\infty\lambda_jb_j$, where $b_j$ is an $(a,q,\alpha)$--atom, for every $j\in \mathbb{N}$.

Actually, the space $\mathcal{H}^{1,q}_a((0,\infty),\gamma_\alpha)$ does not depend on $a$ and $q$. The following result can be proved by proceeding as in \cite[Section~2]{MM} and \cite[Theorem~2.2]{MMS-max}.

\begin{thm}\label{Th1.1} Let $\alpha>-\frac12 $, $a>0$ and $1<q\le \infty$. Then, $\mathcal{H}^{1,q}_a((0,\infty),\gamma_\alpha)$ and~$\mathcal{H}^{1,\infty}_1((0,\infty),\gamma_\alpha)$ coincide, algebraic and topologically.
\end{thm}

In view of the above result, we will simply write $\mathcal{H}^1((0,\infty),\gamma_\alpha)$ to refer to the Hardy space $\mathcal{H}^{1,q}_a((0,\infty),\gamma_\alpha)$,  for any $a>0$ and $1<q\le\infty$.

We now introduce a maximal function we shall use in order to characterize $\mathcal{H}^1((0,\infty),\gamma_\alpha)$. Our definition is inspired in the one given in \cite[p.~273]{MS}.

We consider the measure $\mathfrak{m}_\alpha$ on $(0,\infty)$ by $d\mathfrak{m}_\alpha(x)=x^{2\alpha+1}dx$. By $C_c^1(0,\infty)$ we denote the space of continuously differentiable functions with compact support in~$(0,\infty)$. 

We define, for every $x\in (0,\infty)$, the sets $\mathcal{A}_x^\alpha$ and $\mathcal{A}_{x,\loc}^\alpha$ of functions as follows. Given $x\in (0,\infty)$, a function $\phi\in C^1_c(0,\infty)$ is said to be in $\mathcal{A}_x^\alpha$ when there exists~${0<r\le x}$ such that $\supp(\phi)\subset I(x,r)$ and
\[\mathfrak{m}_\alpha(I(x,r))\|\phi\|_\infty\le 1, \quad r\mathfrak{m}_\alpha(I(x,r))\|\phi'\|_\infty\le 1.\] 
On the other hand, when all of the above hold with some $0<r\le \min\{x,m(x)\}$, $\phi$ is said to be in $\mathcal{A}_{x,\loc}^\alpha$. In other words, $\phi\in \mathcal{A}_{x,\loc}^\alpha$ when $\phi \in \mathcal{A}_x^\alpha$ with $I(x,r)\in \mathcal{B}$. Here, $\|g\|_\infty$ denotes the essential supremum of $g$ in $(0,\infty)$ with respect to the Lebesgue measure (equivalently, to $\mathfrak{m}_\alpha$ or $\gamma_\alpha$). 

We are now in position to define the maximal functions we shall be dealing with. Given $f\in L^1((0,\delta),\mathfrak{m}_\alpha)$ for every $\delta>0$, we set
\[
\mathcal{M}_\alpha(f)(x)=\sup_{\phi\in \mathcal{A}_x^\alpha}\left|\int_0^\infty f(y)\phi(y)\ d\mathfrak{m}_\alpha(y)\right|,\quad x\in (0,\infty),
\]
and
\[
\mathcal{M}_{\alpha,\loc}(f)(x)=\sup_{\phi\in \mathcal{A}_{x,\loc}^\alpha}\left|\int_0^\infty f(y)\phi(y)\ d\mathfrak{m}_\alpha(y)\right|,\quad x\in (0,\infty).
\]

\begin{thm}\label{H1-Malfaloc} Given $\alpha>-\frac12$. A function $f\in L^1((0,\infty),\gamma_\alpha)$ is in $\mathcal{H}^1((0,\infty),\gamma_\alpha)$ if and only if $\mathcal{M}_{\alpha,\loc}(f)\in L^1((0,\infty),\gamma_\alpha)$ and ${\mathcal{E}_\alpha(f):=\int_0^\infty y\left|\int_y^\infty fd\gamma_\alpha\right|dy <\infty}$. Furthermore, for every $f\in \mathcal{H}^1((0,\infty),\gamma_\alpha)$,
\[
\|f\|_{\mathcal{H}^1((0,\infty),\gamma_\alpha)}\simeq \|\mathcal{M}_{\alpha,\loc}(f)\|_{L^1((0,\infty),\gamma_\alpha)}+\mathcal{E}_\alpha(f)<\infty.
\]
\end{thm}

In \cite{Dz} Dziuba\'{n}ski defined Hardy spaces associated with Laguerre function sequences that are basis in $L^2((0,\infty),dx)$. He introduced the admissibility function~$w(x)=\frac18 \min\left\{x,\frac1x\right\}$, $x\in (0,\infty)$ in order to define the atoms. Later, Cha and Li (\cite{ChL1} and \cite{ChL2}) used the function $w$ to define some $\BMO$-type spaces that can be identified with the duals of the Hardy spaces introduced in \cite{Dz}.

Hardy spaces $\mathcal{H}^1((0,\infty),\gamma_\alpha)$ can also be defined using the admissibility function~$w$
instead of the function $m$. The denominator 8 in the function $w$ is not important for us but we prefer to keep the notation given in \cite{ChL1}, \cite{ChL2} and \cite{Dz}.

If $a\in (0,8)$, we say that an interval $I(x,r)$ is in $\mathbb{B}_a$ when $x\in (0,\infty)$ and $0<r\leq aw(x)$. We write $\mathbb{B}:=\mathbb{B}_1$.

Let  $1<q\le \infty$ and $a\in (0,8)$. A measurable function $b$ in $(0,\infty)$ is said to be an \textbf{$(a,q,\alpha)_w$--atom} when $b(x)=1$ for every $x\in (0,\infty)$, or there exist $x,r\in (0,\infty)$ such that $I(x,r)\in \mathbb{B}_a$ and the following properties hold
\begin{enumerate}
    \item $\supp(b)\subset I(x,r)$;
    \item $\|b\|_{L^q((0,\infty),\gamma_\alpha)}\le \gamma_\alpha(I(x,r))^{\frac{1}{q}-1}$;
    \item $\int_0^\infty b(x)\ d\gamma_\alpha(x)=0$.
\end{enumerate}

We now define the atomic Hardy space $\mathbb{H}_a^{1,q}((0,\infty),\gamma_\alpha)$ as follows. We say that a measurable function $f$ defined on $(0,\infty)$ belongs to $\mathbb{H}_a^{1,q}((0,\infty),\gamma_\alpha)$ whenever~${f=\sum_{j=0}^\infty\lambda_jb_j}$, where, for every $j\in \mathbb{N}$, $b_j$ is an $(a,q,\alpha)_w$--atom and $\lambda_j\in \mathbb{C}$ being $\sum_{j=0}^\infty|\lambda_j|<\infty$. For every $f\in \mathbb{H}_a^{1,q}((0,\infty),\gamma_\alpha)$, we define
\[
\|f\|_{\mathbb{H}_a^{1,q}((0,\infty),\gamma_\alpha)}=\inf\sum_{j=0}^\infty|\lambda_j|,
\]
where the infimum is taken over all the sequences $\{\lambda_j\}_{j\in \mathbb{N}}$ of complex numbers for which $\sum_{j=0}^\infty|\lambda_j|<\infty$ and $f= \sum_{j=0}^\infty\lambda_jb_j$ with $(a,q,\alpha)_w$--atoms $b_j$, for every $j\in \mathbb{N}$.

In order to characterize this space by a local maximal function, we shall consider the class of functions $\mathbb{A}$ defined by 
\[
\mathbb{A}=\left\{\phi\in C_c^1(-1,1) :  \|\phi\|_\infty\le 1,\, \|\phi'\|_\infty\le 1\right\}. 
\]
For $\phi\in \mathbb{A}$, we write $\phi_t (x)=\frac{1}{t}\phi\left(\frac{x}{t}\right)$, for $x\in \mathbb{R}$  and $t\in  (0,\infty)$. Note that if $\phi\in \mathbb{A}$, then $\supp(\phi_t (x-\cdot ))\subset I(x,t)\subset (0,\infty)$ provided that $0<t\le x<\infty$. 

For every $a\in (0,8)$, we define the local maximal function $\mathbb{M}_{a,\loc}$ by
\[
\mathbb{M}_{a,\loc}(f)(x)=\sup\{|\phi_t \ast f(x)|:\,\phi\in \mathbb{A}, 0<t<aw(x)\},\quad x\in (0,\infty),
\]
for every $f\in L^1_{\loc}((0,\infty), dx)$, where by $\ast$ we denote the usual convolution in $\mathbb{R}$. 

We also introduce, for every $f\in L^1((0,\infty),\gamma_\alpha)$, the following quantity
\[
\mathbb{E}_\alpha(f)=\int_0^1\frac{1}{y}\left|\int_0^y f(x)\ d\gamma_\alpha(x)\right|dy+\int_1^\infty y\left|\int_y^\infty f(x) \ d\gamma_\alpha(x)\right|dy.
\]

We will prove that $\mathcal{H}^1((0,\infty),\gamma_\alpha)$ can also be characterized by means of $\mathbb{M}_{a,\loc}$ and $\mathbb{E}_\alpha$. Moreover, the atomic spaces  $\mathcal{H}^1((0,\infty),\gamma_\alpha)$ and $\mathbb{H}_a^{1,q}((0,\infty),\gamma_\alpha)$ coincide for any $a\in (0,1]$ and $1<q\leq \infty$.

\begin{thm}\label{equiv-atomicH1} Let $\alpha>-\frac12$, $a\in (0,1]$ and $1<q\le\infty$. For each $f\in L^1((0,\infty),\gamma_\alpha)$, the following statements are equivalent.
\begin{enumerate}[label=(\alph*)]
    \item \label{mathcalH1}  $f\in \mathcal{H}^1((0,\infty),\gamma_\alpha)$;
    \item \label{mathbbH1} $f\in \mathbb{H}_{2a}^{1,q}((0,\infty),\gamma_\alpha)$;
    \item \label{maxlocconv}$\mathbb{M}_{a,\loc}(f)\in L^1((0,\infty),\gamma_\alpha)$ and $\mathbb{E}_\alpha(f)<\infty$.
\end{enumerate}
Moreover, the quantities $\|f\|_{\mathcal{H}^1((0,\infty),\gamma_\alpha)}$, $\|f\|_{\mathbb{H}^{1,q}_a((0,\infty),\gamma_\alpha)}$ and \[\|\mathbb{M}_{a,\loc}(f)\|_{ L^1((0,\infty),\gamma_\alpha)}+\mathbb{E}_\alpha(f)\] are equivalent.
\end{thm}

\begin{rem} From the proof of this theorem (see Section~\ref{sec-equiv-atomiH1}) we shall deduce that the equivalence of properties \ref{mathbbH1} and \ref{maxlocconv} still holds for every $\alpha>-1$.
\end{rem}
\begin{rem}
    The independence of the parameters $a\in(0,1]$ and $1<q\leq\infty$ in the space $\mathbb{H}_{a}^{1,q}((0,\infty),\gamma_\alpha)$ follows from Theorem~\ref{equiv-atomicH1}.
\end{rem}

We now give some definitions and basic properties about Laguerre polynomial expansions and the heat semigroup generated by the Laguerre operator.

Let $\alpha>-\frac12$. 
For every $k\in \mathbb{N}$, the Laguerre polynomial $L_k^\alpha$ of order $\alpha$ and degree~$k$ is defined (see \cite{Leb}) by 
\[
L_k^\alpha(x)=\sqrt{\frac{\Gamma(\alpha+1)}{\Gamma(\alpha+k+1)k!}}e^xx^{-\alpha}\frac{d^k}{dx^k}\left(e^{-x}x^{k+\alpha}\right),\quad x\in (0,\infty).
\]
We consider the Laguerre operator $\widetilde {\Delta}_\alpha$ given by
\[
\widetilde {\Delta}_\alpha(f)(x)=-\frac{1}{4}\left(\frac{d^2}{dx^2}+\left(\frac{2\alpha+1}{x}-2x\right)\frac{d}{dx}\right)f(x),\quad f\in C^2(0,\infty).
\]
By setting, for every $k\in \mathbb{N}$, $\mathcal{L}_k^\alpha(x)=L_k^\alpha(x^2)$, $x\in (0,\infty)$, the sequence
$\{\mathcal{L}_k^\alpha\}_{k\in \mathbb{N}}$ is an orthonormal basis in $L^2((0,\infty),\gamma_\alpha)$. Furthermore, $\widetilde {\Delta}_\alpha\mathcal{L}_k^\alpha=k\mathcal{L}_k^\alpha$,
$k\in \mathbb{N}$.

For every $f\in L^2((0,\infty),\gamma_\alpha)$, we define
\[
c_k^\alpha(f)=\int_0^\infty f(x)\mathcal{L}_k^\alpha(x)d\gamma_\alpha(x),\quad \,k\in \mathbb{N}.
\]
We consider the operator $\Delta_\alpha$ given by
\[
\Delta_\alpha(f)=\sum_{k=0}^\infty kc_k^\alpha(f)\mathcal{L}_k^\alpha, \quad \,f\in D(\Delta_\alpha),
\]
where \[D(\Delta_\alpha)=\left\{f\in L^2((0,\infty),\gamma_\alpha): \sum_{k=0}^\infty|kc_k^\alpha(f)|^2<\infty\right\}.\]

Thus $\Delta_\alpha$ is an extension on $L^2((0,\infty),\gamma_\alpha)$ of $\widetilde{\Delta}_\alpha$ firstly defined on $C_c^\infty(0,\infty)$ (the space of smooth functions with compact support in $(0,\infty)$). The operator $\Delta_\alpha$ is self-adjoint and positive and, moreover, $-\Delta_\alpha$ generates a semigroup of operators $\{W_t^\alpha\}_{t>0}$ in $L^2((0,\infty),\gamma_\alpha)$ where, for every $t>0$, that
\[
W_t^\alpha(f)=\sum_{k=0}^\infty e^{-kt}c_k^\alpha(f)\mathcal{L}_k^\alpha,\quad f\in L^2((0,\infty),\gamma_\alpha).
\]
By using the Hille-Hardy formula (\cite[(4.17.6)]{Leb})
we can write, for every $x,y,t>0$,
\begin{align*}
\sum_{k=0}^\infty e^{-kt}\mathcal{L}_k^\alpha(x)\mathcal{L}_k^\alpha(y)&=\frac{\Gamma(\alpha+1)}{1-e^{-t}}(e^{-t/2}xy)^{-\alpha}I_\alpha\left(\frac{2e^{-t/2}xy}{1-e^{-t}}\right)\\
&\quad \times\exp\left(-\frac{e^{-t}}{1-e^{-t}}(x^2+y^2)\right),    
\end{align*}
where $I_\alpha$ denotes the modified Bessel function of the first kind and order $\alpha$.

We get, for every $f\in L^2((0,\infty),\gamma_\alpha)$ and $t>0$,
\begin{equation}\label{sem}
W_t^\alpha(f)(x)=\int_0^\infty W_t^\alpha(x,y)f(y)d\gamma_\alpha(y),\quad \,x\in (0,\infty),
\end{equation}
being, for every $x,y,t>0$,
\[
W_t^\alpha(x,y)=\frac{\Gamma(\alpha+1)}{1-e^{-t}}(e^{-t/2}xy)^{-\alpha}I_\alpha\left(\frac{2e^{-t/2}xy}{1-e^{-t}}\right)\exp\left(-\frac{e^{-t}}{1-e^{-t}}(x^2+y^2)\right).
\]
The integral in \eqref{sem} is absolutely convergent for every $f\in L^p((0,\infty),\gamma_\alpha)$ and every~$1\le p<\infty$. Moreover, by defining, for every $t>0$, $W_t^\alpha$ on $L^p((0,\infty),\gamma_\alpha)$ with~$1\le p<\infty$, by \eqref{sem} $\{W_t^\alpha\}_{t>0}$ is a symmetric diffusion semigroup in the sense of Stein (\cite{StLP}).

The maximal operator defined by $\{W_t^\alpha\}_{t>0}$ is given by
\[
W_*^\alpha(f)=\sup_{t>0}|W_t^\alpha(f)|.
\]
It is known that $W_{*}^\alpha$ is bounded on $L^p((0,\infty),\gamma_\alpha)$, for every~${1<p<\infty}$ (see \cite[p.~73]{StLP}). Furthermore, from the  Muckenhoupt's results (\cite{Mu2}) it can be deduced that $W_*^\alpha$ is bounded from $L^1((0,\infty),\gamma_\alpha)$ into $L^{1,\infty}((0,\infty),\gamma_\alpha)$.

Suggested by the results in \cite{MVNP} and \cite{Po} we consider the truncated maximal operator $\mathbb{W}_*^\alpha$ defined by
\[
\mathbb{W}_*^\alpha(f)=\sup_{0<t<m(x)^2}|W_t^\alpha(f)|.
\]
Since $\mathbb{W}_*^\alpha(f)\le W_*^\alpha(f)$, $\mathbb{W}_*^\alpha$ is also bounded on $L^p((0,\infty),\gamma_\alpha)$, for every $1<p<\infty$, and from $L^1((0,\infty),\gamma_\alpha)$ into $L^{1,\infty}((0,\infty),\gamma_\alpha)$. Furthermore, we can prove the following endpoint estimate for $\mathbb{W}_*^\alpha$.

\begin{thm}\label{ThH} For any $\alpha\geq0$, the operator $\mathbb{W}_*^\alpha$ is bounded from $\mathcal{H}^1((0,\infty),\gamma_\alpha)$ into $L^1((0,\infty),\gamma_\alpha)$.
\end{thm}

Two open questions related to our results are worth noting:
\begin{enumerate}[label=(\arabic*)]
    \item Is the maximal operator $W_*^\alpha$ bounded from the Hardy space $\mathcal{H}^1((0,\infty),\gamma_\alpha)$ into $L^1((0,\infty),\gamma_\alpha)$?
    \item  Can the Hardy space $\mathcal{H}^1((0,\infty),\gamma_\alpha)$ be characterized by using the maximal operator $\mathbb{W}_*^\alpha$?
\end{enumerate}
As far as we know, these two questions have not been resolved even in the Ornstein-Uhlenbeck setting (see \cite{MVNP} and \cite{Po}).

We will now present the proofs for our Theorems \ref{H1-Malfaloc}, \ref{equiv-atomicH1} and \ref{ThH} in the following sections.

Throughout this paper, by $C$ and $c$ we will always denote positive constants that may change in each occurrence. We will use many times the following properties (see, for instance, \cite[Section~1.2]{U}) without mentioning them: for every $a>0$ there exists $C_a>1$ such that ${\frac{1}{C_a}m(x)\le m(y)\le C_am(x)}$ and $\frac{1}{C_a}\gamma(x)\le \gamma(y)\le C_a\gamma(x)$ provided that $|x-y|\le am(x)$, where $\gamma(x)=e^{-|x|^2}$ for $x\in (0,\infty)$.

\section{Proof of Theorem \ref{H1-Malfaloc}.} \label{sec-H1-Malfaloc}

We now establish some properties that we will use in the proof of Theorem \ref{H1-Malfaloc}.

\begin{lem}\label{L1} Let $\alpha>-\frac12$, $x_0>0$ and $0<r_0\le \min\{x_0, m(x_0)\}$.

Suppose that~$f\in L^1((0,\infty),\gamma_\alpha)$ with $\supp(f)\subset I(x_0,r_0)$. Then, the support of $\mathcal{M}_{\alpha,\loc}(f)$ is contained in $I(x_0, hm(x_0))=(x_0-hm(x_0),x_0+hm(x_0))\cap (0,\infty)$, for a certain $h> 1$ which does not depend on $f$, $x_0$ or $r_0$.
\end{lem}

\begin{proof} We fix $x>0$ and take $\phi\in \mathcal{A}_{x,\loc}^\alpha$ with $\supp(\phi)\subset I(x,r)$, for some ${0<r\leq \min\{x,m(x)\}}$. Clearly, we can assume $\supp(\phi)\cap I(x_0,r_0)\neq\emptyset$ since, otherwise, $\int_0^\infty \phi(y)f(y)y^{2\alpha+1}dy=0$. Choosing $z\in \supp(\phi)\cap I(x_0,r_0)$, it follows that
$m(x)\sim m(z)\sim m(x_0)$, from which yields that $x\in I(x_0, hm(x_0))$, for some $h> 1$. Thus, we have proved that $\supp(\mathcal{M}_{\alpha,\loc}(f))\subset  I(x_0, hm(x_0))$.
\end{proof}

We consider, for every $j\in \mathbb{N}\setminus\{0\}$, $I_j=(\sqrt{j-1},\sqrt{j+1})$ and $I_0=(0,1)$. The center and the radius of $I_j$ are $c_j=\frac{1}{2}(\sqrt{j+1}+\sqrt{j-1})$ and $r_j=1/(2c_j)$, respectively, for every $j\in \mathbb{N}\setminus\{0\}$, and the center and radius of $I_0$ are $c_0=1/4$ and $r_0=1/4$, respectively. Then, $I_j\in \mathcal{B}$ for each $j\in \mathbb{N}$. We also define $\mathcal{I}_j=I_j$,  $\mathcal{I}_{-j}=(-\sqrt{j+1},-\sqrt{j-1})$, $j\in \mathbb{N}\setminus\{0\}$, and $\mathcal{I}_0=(-1,1)$. We consider a partition of unity $\{\hat{\eta}_j\}_{j\in \mathbb{N}}$ associated with  $\{\mathcal{I}_j\}_{j\in \mathbb{Z}}$ (see \cite[p.~1685]{MMS-max}). We define $\eta_j=\hat{\eta}_j$, $j \in \mathbb{N}$, $j\ge 1$, and $\eta_0=\left.(\hat{\eta}_0)\right|_{I_0}$.

\begin{lem}\label{L2} Let $\alpha>-\frac12 
$. There exist $h,C>0$ such that, for every $x\in (0,\infty)$ and $j\in \mathbb{N}$, and any $g\in L^1((0,\delta),\gamma_\alpha)$, for each $\delta>0$,
\[
\mathcal{M}_{\alpha,\loc}(g\eta_j\gamma)(x)\le C\gamma(c_j)\mathcal{M}_{\alpha,\loc}(g)(x)\chi_{hI_j}(x).
\]
\end{lem}

\begin{proof} Let  $g\in L^1((0,\delta),\gamma_\alpha)$, for each $\delta>0$, $x\in (0,\infty)$ and $j\in \mathbb{N}$. Assume that $\phi\in \mathcal{A}_{x,\loc}^\alpha$ and define $\psi=\phi\eta_j\gamma/\gamma(c_j)$. We have that $\psi\in C^1_c(0,\infty)$ and
\begin{enumerate}
    \item $ \supp(\psi)\subset \supp(\phi)\subset I(x,r)$, for some $0<r\le \min\{x,m(x)\}$;
    \item $\mathfrak{m}_\alpha(I(x,r))\|\psi\|_\infty\le C\mathfrak{m}_\alpha(I(x,r))\|\phi\|_\infty\le C$;
    \item since $|\eta|\le 1$, $r_j|\eta_j'|\le C$, and $0<r\le 1$, we have that
\begin{align*}
r\mathfrak{m}_\alpha(I(x,r))\psi'(y)&\le \frac{r\mathfrak{m}_\alpha(I(x,r))}{\gamma(c_j)}\left(|\phi'(y)|\eta_j(y)|\gamma(y)\right.\\
&\quad \left. +|\phi(y)||\eta_j'(y)|\gamma(y)+2|\phi(y)||\eta_j(y)|y\gamma(y)\right)\\
&\le C,
\end{align*}
for any $y\in (0,\infty)$. Indeed, assume that $I(x,r)\cap I_j\neq\emptyset$. By taking $z\in I(x,r)\cap I_j$ we deduce that  $m(x)\sim m(z)\sim m(c_j)$. Furthermore, if $y\in I_j$, then
\begin{align*}
  ry&=r(y-c_j+c_j)\le r(r_j+c_j)\\
  &\le m(x)(r_j+c_j)\le Cm(c_j)(r_j+c_j)\le C.  
\end{align*}
\end{enumerate}
Also, according to Lemma \ref{L1}, $\supp(\mathcal{M}_{\alpha,\loc}(g\eta_j\gamma))\subset hI_j$ for some $h\geq 1$. With this, we obtain the desired result.
\end{proof}

\begin{lem}\label{L3} Let $f\in L^1((0,\infty),\gamma_\alpha)$ and $j\in \mathbb{N}$. We define
\[
b_j=\frac{\int f\eta_jd\gamma_\alpha}{\int\eta_jd\gamma_\alpha}.
\]
Suppose that $\mathcal{M}_{\alpha,\loc}(f)\in L^1((0,\infty),\gamma_\alpha)$. Then, $\mathcal{M}_{\alpha}((f-b_j)\eta_j\gamma)\in L^1((0,\infty),\mathfrak{m}_\alpha)$ with
\[
\|\mathcal{M}_{\alpha}((f-b_j)\eta_j\gamma)\|_{L^1((0,\infty),\mathfrak{m}_\alpha)}\le C\int_{hI_j}\mathcal{M}_{\alpha,\loc}(f)d\gamma_\alpha,
\]
where $C,h>0$ are are independent of $f$ and $j$.
\end{lem}

\begin{proof} Since we can split
\begin{align*}
    \mathcal{M}_{\alpha}&((f-b_j)\eta_j\gamma)(x)\\
    &=\mathcal{M}_{\alpha,\loc}((f-b_j)\eta_j\gamma)(x)+\sup_{\phi\in \mathcal{A}_x^\alpha \setminus \mathcal{A}_{x,\loc}^\alpha}\left|\int_0^\infty (f-b_j)\eta_j\gamma(y)\phi(y)\ d\mathfrak{m}_\alpha(y)\right|,
\end{align*}
for any $x\in (0,\infty)$, we will estimate their norms in $L^1((0,\infty),\mathfrak{m}_\alpha)$ separately.

For the first term, according to Lemma~\ref{L2}, applied twice, we get
\begin{align*}
\int_0^\infty &\mathcal{M}_{\alpha,\loc}((f-b_j)\eta_j\gamma)(x)d\mathfrak{m}_\alpha(x)\\
&\le \int_0^\infty \mathcal{M}_{\alpha,\loc}(f\eta_j\gamma)(x)d\mathfrak{m}_\alpha(x)+|b_j|\int_0^\infty \mathcal{M}_{\alpha,\loc}(\eta_j\gamma)(x)d\mathfrak{m}_\alpha(x)\\
&\le C\left(\int_{hI_j}\mathcal{M}_{\alpha,\loc}(f)(x)d\mathfrak{m}_\alpha(x)\gamma(c_j)+|b_j|\int_{hI_j}\mathcal{M}_{\alpha,\loc}(1)(x)d\mathfrak{m}_\alpha(x)\gamma(c_j)\right)\\
&\le C\left(\int_{hI_j}\mathcal{M}_{\alpha,\loc}(f)(x)d\gamma_\alpha(x)+|b_j|\gamma_\alpha(hI_j)\right)
\end{align*}
Since $|f|\le \mathcal{M}_{\alpha,\loc}(f)$ and
\[
|b_j|=\left|\frac{\int f\eta_jd\gamma_\alpha}{\int \eta_jd\gamma_\alpha}\right|\le \frac{1}{\gamma_\alpha\left(\frac12 I_j\right)}\int_{I_j}|f|d\gamma_\alpha\le  C\frac{1}{\gamma_\alpha(I_j)}\int_{I_j}|f|d\gamma_\alpha,
\]
we obtain
\begin{align*}
\int_0^\infty &\mathcal{M}_{\alpha,\loc}((f-b_j)\eta_j\gamma)(x)d\mathfrak{m}_\alpha(x)\\
&\le C\left(\int_{hI_j}\mathcal{M}_{\alpha,\loc}(f)(x)d\gamma_\alpha(x)+\frac{\gamma_\alpha(hI_j)}{\gamma_\alpha(I_j)}\int_{I_j}|f|d\gamma_\alpha(x)\right)\\
&\le C\int_{hI_j}\mathcal{M}_{\alpha,\loc}(f)(x)d\gamma_\alpha(x).
\end{align*}
Suppose now that $\phi\in \mathcal{A}_x^\alpha\setminus\mathcal{A}_{x,\loc}^\alpha$, with $x\in hI_j$. By recalling the notation $I(x,m(x)):=(x-m(x),x+m(x))\cap (0,\infty)$, we have that $\supp(\phi)\subset I(x,r)$ for some $0<r\leq x$ with $r>m(x)$, and $\mathfrak{m}_\alpha(I(x,m(x)))\|\phi\|_\infty \leq 1$. Hence,
\begin{align*}
\left|\int_0^\infty (f-b_j)\phi\eta_j\gamma d\mathfrak{m}_\alpha\right|&\le \frac{C}{\mathfrak{m}_\alpha(I(x, m(x)))}\int_{I_j}|f-b_j|d\gamma_\alpha\\
&\le\frac{C}{\mathfrak{m}_\alpha(I(x, m(x)))}\left(\int_{I_j}|f|d\gamma_\alpha+|b_j|\gamma_\alpha(I_j)\right)\\
&\le \frac{C}{\mathfrak{m}_\alpha(I(x, m(x)))}
\int_{I_j}|f|d\gamma_\alpha.
\end{align*}

It follows that
\begin{align*}
\int_{hI_j}\sup_{\phi\in \mathcal{A}_x^\alpha\setminus\mathcal{A}_{x,\loc}^\alpha}& \left|\int_0^\infty (f-b_j)\phi\eta_j\gamma d\mathfrak{m}_\alpha\right|d\mathfrak{m}_\alpha(x)\\
&\le C\int_{I_j}|f|d\gamma_\alpha\int_{hI_j}\frac{d\mathfrak{m}_\alpha(x)}{\mathfrak{m}_\alpha(I(x, m(x)))}.
\end{align*}

Since $\mathfrak{m}_\alpha$ is doubling  and $m(x)\sim m(c_j)$ provided that $x\in hI_j$, where the equivalence constant is not depending on $j$, we obtain
\begin{align*}
\int_{hI_j}\sup_{\phi\in \mathcal{A}_x^\alpha\setminus\mathcal{A}_{x,\loc}^\alpha}& \left|\int_0^\infty (f-b_j)\phi\eta_j\gamma d\mathfrak{m}_\alpha\right|d\mathfrak{m}_\alpha(x)\\
&\le C\int_{I_j}|f|d\gamma_\alpha\int_{hI_j}\frac{d\mathfrak{m}_\alpha(x)}{\mathfrak{m}_\alpha(I(x, Cm(c_j)))}\\
&\le C\int_{I_j}|f|d\gamma_\alpha\int_{hI_j}\frac{d\mathfrak{m}_\alpha(x)}{\mathfrak{m}_\alpha(I_j)}\le C\int_{I_j}|f|d\gamma_\alpha.
\end{align*}

We now study
\[
\int_{(hI_j)^c}\sup_{\phi\in \mathcal{A}_x^\alpha\setminus\mathcal{A}_{x,\loc}^\alpha} \left|\int_0^\infty (f-b_j)\phi\eta_j\gamma d\mathfrak{m}_\alpha\right|d\mathfrak{m}_\alpha(x).
\]
Let $x\in (hI_j)^c$ and $\phi\in \mathcal{A}_x^\alpha\setminus\mathcal{A}_{x,\loc}^\alpha$. Assume that $\int_0^\infty(f-b_j)\phi\eta_j\gamma d\mathfrak{m}_\alpha\neq 0$. Then, $\supp(\phi)\cap I_j\neq \emptyset$ which yields that $\supp(\phi)\subset I(x,r)$, with $r>\dist (x,I_j)$. Since $\int_0^\infty (f-b_j)\eta_jd\gamma_\alpha=0$ we can write
\begin{align*}
&\left|\int_0^\infty(f-b_j)\phi\eta_j\gamma d\mathfrak{m}_\alpha\right|\\
&=C\left|\int_0^\infty(\phi(y)-\phi(c_j))(f(y)-b_j)\eta_j(y)\gamma(y) d\mathfrak{m}_\alpha(y)\right|\\
&\le C\int_{I_j}|\phi(y)-\phi(c_j)||f(y)-b_j||\eta_j(y)|d\gamma_\alpha(y)\\
&\le C\frac{r_j}{r\mathfrak{m}_\alpha(I(x,r))}\int_{I_j}|f(y)-b_j|d\gamma_\alpha(y)\\
&\le C\frac{r_j}{\dist(x,I_j)\mathfrak{m}_\alpha(I(x,\dist(x,I_j)))}\int_{I_j}|f(y)|d\gamma_\alpha(y).
\end{align*}

We choose $k_0\in \mathbb{N}$ such that $2^{k_0}\le h<2^{k_0+1}$. Thus, calling $R_k(c_j,r_j)= I(c_j,2^{k+1}r_j)\setminus I(c_j,2^{k}r_j)$,
\begin{align*}
\int_{(hI_j)^c}&\frac{1}{\dist (x,I_j)\mathfrak{m}_\alpha(I(x,\dist (x,I_j)))}d\mathfrak{m}_\alpha(x)\\
&\leq C\sum_{k=k_0}^\infty \int_{R_k(c_j,r_j)}\frac{x^{2\alpha+1}}{\dist (x,I_j)\mathfrak{m}_\alpha(I(x,\dist (x,I_j)))}dx\\
&\le C\sum_{k=k_0}^\infty \int_{R_k(c_j,r_j)}\frac{1}{\dist (x,I_j)^2}dx\\
&\le C\sum_{k=0}^\infty\int_{2^kr_j+c_j}^{2^{k+1}r_j+c_j}\frac{dx}{(2^kr_j)^2}\\
&\le \frac{C}{r_j}\sum_{k=0}^\infty 2^{-k}=\frac{C}{r_j},
\end{align*}
where, in the first inequality, we have used that $\mathfrak{m}_\alpha(I(z,r))\ge C z^{2\alpha+1}r$, for every $z,r>0$ (see \cite[(1.4)]{YY}).

Therefore
\[
\int_{(hI_j)^c}\sup_{\phi\in \mathcal{A}_x^\alpha\setminus\mathcal{A}_{x,\loc}^\alpha} \left|\int_0^\infty (f-b_j)\phi\eta_j\gamma d\mathfrak{m}_\alpha\right|d\mathfrak{m}_\alpha(x)\le C\int_{I_j}|f|d\gamma_\alpha.
\]
and we can conclude that
\[\int_0^\infty \mathcal{M}_{\alpha}((f-b_j)\eta_j\gamma)d\mathfrak{m}_\alpha\le C\int_{hI_j}\mathcal{M}_{\alpha,\loc}(f)d\gamma_\alpha.\qedhere\]
\end{proof}

Suppose that $f\in L^1((0,\infty),\gamma_\alpha)$ with $\int_0^\infty fd\gamma_\alpha=0$, $\mathcal{M}_{\alpha,\loc}(f)\in L^1((0,\infty),\gamma_\alpha)$ and $\mathcal{E}_\alpha(f)<\infty$. If $\int_0^\infty fd\gamma_\alpha\neq 0$ we apply our result to $f-\int_0^\infty fd\gamma_\alpha$. To prove the ``if'' implication of Theorem~\ref{H1-Malfaloc} we are going to see that, for some $a>0$, $f=\sum_{j=0}^\infty\lambda_jb_j$, where, for every $j\in \mathbb{N}$, $b_j$ is an $(a,\infty,\alpha)$-atom, and $\lambda_j\in \mathbb{C}$ being
\[
\sum_{j=0}^\infty|\lambda_j|\le C\left(\int_0^\infty|\mathcal{M}_{\alpha,\loc}f(y)|d\gamma_\alpha(y)+\mathcal{E}_\alpha(f)\right).
\]

To achieve this goal, we are going to consider the Hardy space $H^1_\alpha(0,\infty)$ associated with the Bessel operator \[S_\alpha=-\frac{d^2}{dx^2}-\frac{2\alpha+1}{x}\frac{d}{dx},\] that is, the Hardy space defined by the heat semigroup generated by $S_\alpha$ (see \cite{BDT} and \cite{YY}). We shall also take into account the following atoms: we say that a measurable function $b$ is a \textbf{$(\mathfrak{m}_\alpha,\infty)$-atom} when there exist $x_0,\,r_0\in (0,\infty)$ such that
\begin{enumerate}[label=(\alph*)]
    \item $\supp(b)\subset I(x_0, r_0)$;
    \item $\|b\|_\infty\le \mathfrak{m}_\alpha(I(x_0, r_0))^{-1}$;
    \item $\int_0^\infty b(y)\ d\mathfrak{m}_\alpha(y)=0$.
\end{enumerate}

As it was proved in \cite{BDT}, the space $H^1_\alpha(0,\infty)$ coincides with the atomic Hardy space $H^1((0,\infty),|\cdot|,\mathfrak{m}_\alpha))$ defined through these $(\mathfrak{m}_\alpha,\infty)$-atoms in $((0,\infty),|\cdot|,\mathfrak{m}_\alpha)$, in the sense of Coifman and Weiss (\cite{CW}). Clearly, the triple $((0,\infty),|\cdot|,\mathfrak{m}_\alpha)$  is a space of homogeneous type. According to \cite[p.~201]{BDT} the space $H^1_\alpha((0,\infty)$ also coincides with the Coifman-Weiss atomic Hardy space $H^1((0,\infty),d_\alpha,\mathfrak{m}_\alpha)$ associated with the normal homogeneous space $((0,\infty),d_\alpha,\mathfrak{m}_\alpha)$. Here $d_\alpha$ represents the metric defined on $(0,\infty)$ by $d_\alpha(x,y)=\frac{1}{2\alpha+2}\left|x^{2\alpha+2}-y^{2\alpha+2}\right|$, $x,y\in (0,\infty)$.

The following lemma gives us a decomposition for any function in  $H^1_\alpha(0,\infty)$ supported on an interval $I$ into $(\mathfrak{m}_\alpha,\infty)$-atoms with controlled support.

\begin{lem} \label{L31} Let $\alpha>-\frac12 $. There exists $C>0$ such that, for every interval $I\subset (0,\infty)$ and $g\in H^1_\alpha(0,\infty)$ such that $\supp(g)\subset I$, we can write $g=\sum_{j=0}^\infty\lambda_jb_j$ in $L^1((0,\infty),\mathfrak{m}_\alpha)$, where, for every $j\in \mathbb{N}$, $b_j$ is a $(\mathfrak{m}_\alpha,\infty)$-atom having its support contained in $2I$, and $\lambda_j\in \mathbb{C}$ being $\sum_{j=0}^\infty|\lambda_j|\le C\|g\|_{H^1_\alpha(0,\infty)}$.
\end{lem}

\begin{proof} If $I=(0,\infty)$ the property is clear from \cite[Theorem~1.7]{BDT}. Let $I$ be a bounded interval of $(0,\infty)$ and suppose that $g\in H^1_\alpha(0,\infty)$ with $\supp(g)\subset I$. Then, according to \cite[Theorem~4.13]{MS} by introducing a cut off function with respect to the intervals $I$ and $2I$ in the maximal functions we deduce that $g\in H^1(2I,d_\alpha,\mathfrak{m}_\alpha)$ and
\begin{equation}\label{H10}
\|g\|_{H^1(2I,d_\alpha,\mathfrak{m}_\alpha)}\le C\|g\|_{H^1((0,\infty),d_\alpha,\mathfrak{m}_\alpha)}\le C\|g\|_{H^1_\alpha(0,\infty)}.
\end{equation}
Therefore, we can write $g=\sum_{j=1}^\infty\lambda_jb_j+\lambda_0b_0$ in $L^1((0,\infty),\mathfrak{m}_\alpha)$, where $b_j$ is an~$(\mathfrak{m}_\alpha,\infty)$-atom such that $\supp(b_j)\subset 2I$, for every $j\in \mathbb{N}\setminus\{0\}$, $b_0(x)=1$ for every $x\in 2I$, and $\lambda_j\in \mathbb{C}$, for every $j\in \mathbb{N}$, being \begin{equation}\label{H11}
\sum_{j=0}^\infty|\lambda_j|\le C\|g\|_{H^1_\alpha(0,\infty)}.
 \end{equation}
 Since $\int_0^\infty g(y)d\mathfrak{m}_\alpha(y)=\int_0^\infty b_j(y)d\mathfrak{m}_\alpha(y)=0$, for every $j\in \mathbb{N}\setminus\{0\}$, and the series $\sum_{j=1}^\infty\lambda_jb_j+\lambda_0b_0$ converges in $L^1((0,\infty),\mathfrak{m}_\alpha)$, it follows that $\lambda_0=0$. Note that the constants $C>0$ in \eqref{H10} and \eqref{H11} do not depend on $g$ or $I$.
\end{proof}

The following result will allow us to obtain an atomic decomposition of a function $g$ into $(\mathfrak{m}_\alpha,\infty)$-atoms whenever $\mathcal{M}_\alpha(g)\in L^1((0,\infty),\mathfrak{m}_\alpha)$.

\begin{lem}\label{L32}
Let $\alpha>-\frac12 $. There exists $C>0$ such that if $g\in L^1((0,\infty),\mathfrak{m}_\alpha)$, then $g\in H^1_\alpha(0,\infty)$ provided that $\mathcal{M}_\alpha(g)\in L^1((0,\infty),\mathfrak{m}_\alpha)$ being
\[
\|g\|_{H^1_\alpha(0,\infty)}\le C\|\mathcal{M}_\alpha(g)\|_{L^1((0,\infty),\mathfrak{m}_\alpha)}.
\]
\end{lem}

\begin{proof} Assume that $0\le\phi\in C_c^\infty(0,\infty)$ and $\supp(\phi)\subset (0,1)$. We define, for every $t>0$, $\phi_t(x)=t^{-2\alpha-2}\phi(x/t)$, $x\in (0,\infty)$, and, for every $t,x\in (0,\infty)$, $\Phi_{x,t}={}_{\alpha}\tau_x(\phi_t)$. Here ${}_{\alpha}\tau_x$ denotes the $x$-translation associated with the operator $S_\alpha$, also called Hankel $x$-translation (see \cite{GS}, \cite{Ha} and \cite{Hi}). We can write
\[
\Phi_{x,t}(y)=\int_{|x-y|}^{x+y}\phi_t(z)D_\alpha(x,y,z)z^{2\alpha+1}dz,\quad \,x,y,t\in (0,\infty),
\]
where, for every $x,y,z\in (0,\infty)$,
\[
D_\alpha(x,y,z)=\frac{2^{2\alpha-1}\Gamma(\alpha+1)}{\Gamma\left(\alpha+\frac12\right)\sqrt{\pi}}(xyz)^{-2\alpha}(\Upsilon(x,y,z))^{2\alpha-1},\quad x,y,z\in (0,\infty),
\]
being $\Upsilon(x,y,z)$ the area of a triangle with side lengths $x,\,y,\,z$ when this triangle exists and $\Upsilon(x,y,z)=0$ in other cases. It follows that $\Phi_{x,t}(y)=0$ for every $|x-y|>t$ and $t>0$.

We also have that, for any $x,y,t\in (0,\infty)$,
\begin{equation}\label{Hankel-integral}
\Phi_{x,t}(y)=\frac{\Gamma(\alpha+1)}{\Gamma\left(\alpha+\frac12\right)\sqrt{\pi}}\int_0^\pi\phi_t\left(\sqrt{x^2+y^2-2xy\cos\theta}\right)(\sin\theta)^{2\alpha}d\theta.    
\end{equation}

Notice that, since $\phi\in C_c^\infty(0,\infty)$, for every $x\in (0,\infty)$ it satisfies the following properties
\begin{enumerate}[label=(\alph*)]
    \item $0\le \phi(x)\le C(1+x^2)^{-\alpha-3/2}$;
    \item $|\phi'(x)|\le Cx(1+x^2)^{-\alpha-5/2}$;
    \item $|\phi''(x)|\le C(1+x^2)^{-\alpha-5/2}$.
\end{enumerate}
 Then, according to \cite[Theorem~2.7]{BDT} and \cite[Theorem~1.1]{YY}, $g$ is in $H^1_\alpha(0,\infty)$ provided that $g\in L^1((0,\infty),\mathfrak{m}_\alpha)$
  and 
  \[
\sup_{t>0}\left|\int_0^\infty\Phi_{x,t}(y)g(y)d\mathfrak{m}_\alpha(y)\right|\in L^1((0,\infty),\mathfrak{m}_\alpha).
\]
Since $\mathcal{M}_\alpha(g)\in L^1((0,\infty),\mathfrak{m}_\alpha)$,
  it is enough to show that for every $g\in L^1_{\loc}(0,\infty)$,
\begin{equation}\label{eq: maximales}
\sup_{t>0}\left|\int_0^\infty\Phi_{x,t}(y)g(y)d\mathfrak{m}_\alpha(y)\right|\le \mathcal{M}_\alpha(g)(x),\quad x\in (0,\infty).
\end{equation}
We will do so by showing that $\Phi_{x,t}\in \mathcal{A}_{x}^\alpha$ for every $x,t\in (0,\infty)$.

First, we are going to see that
\begin{equation}\label{Z1}
\left|\Phi_{x,t}(y)\right|\le \frac{C}{\mathfrak{m}_\alpha(I(x,t))},\quad x,y,t\in (0,\infty).
\end{equation}
By taking $0<\epsilon<2\alpha+1$, and using that $2(1-\cos\theta)\sim \theta^2$ for every $\theta\in [0,\pi]$, we have that
\begin{align*}
|\Phi_{x,t}(y)|&\le \frac{C}{t^{2\alpha+2}}\int_0^\pi \left|\phi\left(\frac{\sqrt{x^2+y^2-2xy\cos\theta}}{t}\right)\right|(\sin\theta)^{2\alpha}d\theta\\
&\le \frac{C}{t^{2\alpha+2}}\left(\int_0^{t/\sqrt{xy}}\theta^{2\alpha}\left(\frac{t}{\sqrt{(x-y)^2+xy\theta^2}}\right)^{2\alpha+1-\epsilon}d\theta \right.\\
&\left.+\int_{t/\sqrt{xy}}^\pi\theta^{2\alpha}\left(\frac{t}{\sqrt{(x-y)^2+xy\theta^2}}\right)^{2\alpha+1+\epsilon}d\theta\right)\\
&\le C\left(\frac{1}{t^{1+\varepsilon}}\int_0^{t/\sqrt{xy}}\theta^{\epsilon-1}d\theta (xy)^{-\alpha-(1-\epsilon)/2}\right.\\
&\left.+\frac{1}{t^{1-\varepsilon}}\int_{t/\sqrt{xy}}^\infty\theta^{-\epsilon-1}d\theta (xy)^{-\alpha-(1+\epsilon)/2}\right)\\
&\le C\frac{1}{t(xy)^{\alpha+1/2}},\quad x,y,t\in (0,\infty).
\end{align*}
Then, whenever $0<\frac x2<y<2x<\infty$
\[
\left|\Phi_{x,t}(y)\right|\le\frac{C}{tx^{2\alpha+1}},\quad t\in (0,\infty).
\]
On the other hand, if $0<y<\frac x2<y$ or $2x<y<\infty$, we get
\begin{align*}
\left|\Phi_{x,t}(y)\right|&\le\frac{C}{t^{2\alpha+2}}\int_0^\pi\left(\frac{t}{\sqrt{x^2+y^2-2xy\cos\theta}}\right)^{2\alpha+1}\theta^{2\alpha}d\theta\\
&\le \frac{C}{t|x-y|^{2\alpha+1}}\\
&\le \frac{C}{tx^{2\alpha+1}},\quad t\in (0,\infty).
\end{align*}
Since $\mathfrak{m}_\alpha(I(x,t))\sim t(x+t)^{2\alpha+1}$, for any $x,t\in (0,\infty)$ and, as it is clear from~\eqref{Hankel-integral},
\[
\left|\Phi_{x,t}(y)\right|\le \frac{C}{t^{2\alpha+2}},\quad x,y,t\in (0,\infty),
\]
we conclude that \eqref{Z1} holds.

Now, for every $x,y,t\in (0,\infty)$,
\begin{align*}
\partial_y\Phi_{x,t}(y)=\frac{\Gamma(\alpha+1)}{\Gamma\left(\alpha+\frac12\right)\sqrt{\pi}t^{2\alpha+3}}\int_0^\pi&\phi'\left(\frac{\sqrt{x^2+y^2-2xy\cos\theta}}{t}\right)\\
& \times \frac{y-x\cos\theta}{\sqrt{x^2+y^2-2xy\cos\theta}}(\sin\theta)^{2\alpha}d\theta.
\end{align*}
Since, for every $x,y\in (0,\infty)$ and $\theta\in (0,\pi)$,
\[
(y-x\cos\theta)^2=y^2+x^2(\cos\theta)^2-2xy\cos\theta=x^2+y^2-2xy\cos\theta-(1-\cos^2\theta)x^2,
\]
we get
\begin{align*}
|\partial_y\Phi_{x,t}(y)|&\le \frac{\Gamma(\alpha+1)}{\Gamma\left(\alpha+\frac12\right)\sqrt{\pi}t^{2\alpha+3}}\int_0^\pi\phi'\left(\frac{\sqrt{x^2+y^2-2xy\cos\theta}}{t}\right)(\sin\theta)^{2\alpha}d\theta\\
&=\frac{1}{t}\, {}_{\alpha}\tau_x(|\phi'|)(y)\\
&\le \frac{1}{t\mathfrak{m}_\alpha(I(x,t))},\quad \,x,y,t\in (0,\infty).
\end{align*}
In order to see the last inequality we can proceed as above {by considering $|\phi'|$ instead of $\phi$}. Thus, we have proved~\eqref{eq: maximales} and this finishes the proof.
\end{proof}

Hence, according to Lemmas \ref{L3}, \ref{L31}, and \ref{L32}, for every $j\in \mathbb{N}$ we obtain that~${(f-b_j)\eta_j\gamma\in H^1_\alpha(0,\infty)}$, and there exist a sequence $\{\lambda_{i,j}\}_{i\in \mathbb{N}}\subset \mathbb{C}$ and a sequence $\{a_{i,j}\}_{i\in \mathbb{N}}$ of $(\mathfrak{m}_\alpha,\infty)$-atoms such that $\supp(a_{i,j})\subset 2I_j$, for every $i\in \mathbb{N}$, and
\[
(f-b_j)\eta_j\gamma=\sum_{i=0}^\infty\lambda_{i,j}a_{i,j},
\]
being $\sum_{i=0}^\infty|\lambda_{i,j}|\le C\int_{hI_j}\mathcal{M}_{\alpha,\loc}(f)d\gamma_\alpha$.

Moreover, for each $j\in \mathbb{N}$,
\begin{align*}
\sum_{i=0}^\infty|\lambda_{i,j}|\|a_{i,j}/\gamma\|_{L^1((0,\infty),\gamma_\alpha)}&=\sum_{j=0}^\infty|\lambda_{i,j}|\|a_{i,j}\|_{L^1((0,\infty),\mathfrak{m}_\alpha)}\\
&\le C\sum_{i=0}^\infty|\lambda_{i,j}|\\
&\le C\int_{hI_j}\mathcal{M}_{\alpha,\loc}(f)d\gamma_\alpha.
\end{align*}
Then,
\begin{align*}
\sum_{j=0}^\infty\sum_{i=0}^\infty|\lambda_{i,j}|\|a_{i,j}/\gamma\|_{L^1((0,\infty),\gamma_\alpha)}&\le C\sum_{j=0}^\infty\int_{hI_j}\mathcal{M}_{\alpha,\loc}(f)d\gamma_\alpha\\
&\le C\|\mathcal{M}_{\alpha,\loc}(f)\|_{L^1(0,\infty),\gamma_\alpha)},
\end{align*}
so the double series $\sum_{(i,j)\in \mathbb{N}\times\mathbb{N}}|\lambda_{i,j}||a_{i,j}/\gamma|$ converges in $L^1((0,\infty),\gamma_\alpha)$. Then, since the last series consists of positive functions, the series $\sum_{(i,j)\in \mathbb{N}\times\mathbb{N}}|\lambda_{i,j}||a_{i,j}/\gamma|$ converges almost everywhere in $(0,\infty)$ and we can write
\[
\sum_{(i,j)\in \mathbb{N}\times\mathbb{N}}\lambda_{i,j}a_{i,j}/\gamma=\sum_{j=0}^\infty\sum_{i=0}^\infty\lambda_{i,j}a_{i,j}/\gamma
\]
in $L^1((0,\infty),\gamma_\alpha)$ and almost everywhere in $(0,\infty)$.

Furthermore, for every $i,j\in \mathbb{N}$, $a_{i,j}/\gamma$ is a $(1,\infty,\alpha)$-atom as ${\supp(a_{i,j})\subset 2I_j}$. We have then just proved that
\[
\sum_{j=0}^\infty (f-b_j)\eta_j=\sum_{j=0}^\infty\sum_{i=0}^\infty\lambda_{i,j}a_{i,j}/\gamma.
\]

Our next objective is to study $\sum_{j=0}^\infty b_j\eta_j$. Since
\[
|b_j|\le \frac{C}{\gamma_\alpha(I_j)}\int_{I_j}|f|d\gamma_\alpha,\quad j\in \mathbb{N},
\]
we get
\[
\left\|\sum_{j=0}^\infty |b_j||\eta_j|\right\|_{L^1((0,\infty),\gamma_\alpha)}\le C\sum_{j=0}^\infty \int_{I_j}|f|d\gamma_\alpha\le C\|f\|_{L^1((0,\infty),\gamma_\alpha)}.
\]
Then, there exists an increasing function $\psi:\mathbb{N}\rightarrow \mathbb{N}$ such that
\[
\lim_{k\to\infty}\sum_{j=0}^{\psi(k)}|b_j(x)||\eta_j(x)|
\]
exists for almost every $x\in (0,\infty)$. We conclude that the series $\sum_{j=0}^\infty |b_j||\eta_j|$ converges almost everywhere in $(0,\infty)$.

For $j\in \mathbb{N}$, we define $\widetilde {\eta}_j=\eta_j/\int_0^\infty\eta_jd\gamma_\alpha$, so we can write $b_j\eta_j=\widetilde {\eta}_j\int_0^\infty f\eta_jd\gamma_\alpha$. We also consider $\mu_k=\sum_{j=k}^\infty\eta_j$, for $k\in \mathbb{N}$. Since
\[
\int_0^\infty |f(y)|\sum_{j=k}^\infty\eta_j(y)d\gamma_\alpha(y)\le C\int_0^\infty |f(y)|d\gamma_\alpha(y), \quad k\in \mathbb{N},
\]
it follows that
\[
\sum_{j=k}^\infty \int_0^\infty f(y)\eta_j(y)d\gamma_\alpha=\int_0^\infty f(y)\mu_k(y)d\gamma_\alpha(y),\quad k\in \mathbb{N}.
\]
We get
\[
\sum_{j=0
^\infty} b_j\eta_j=\left|\int_0^\infty f(y)\mu_k(y)d\gamma_\alpha(y)\right|\le \sum_{j=k}^\infty\int_{I_j}|f(y)|d\gamma_\alpha(y)\le C\int_{\sqrt{k-1}}^\infty|f(y)|d\gamma_\alpha(y)
\]
where the last integral tends to zero as $k\to \infty$. By summation by parts, since $\lim_{k\to\infty}\eta_k(x)=0$, $x\in (0,\infty)$, we deduce
\[\sum_{j=0}^\infty b_j\eta_j=
\sum_{j=0}^\infty\int_0^\infty f(y)\eta_j(y)d\gamma_\alpha(y)\widetilde {\eta}_j=\sum_{k=0}^\infty \int_0^\infty f(y)\mu_{k+1}(y)d\gamma_\alpha(y)(\widetilde {\eta}_{k+1}-\widetilde {\eta}_{k}).
\]
We claim there exist $C>0$ such that $C(\widetilde {\eta}_{k+1}-\widetilde {\eta}_{k})$  is a $(2,\infty,\alpha)$-atom, for every $k\in \mathbb{N}$. Indeed
\begin{enumerate}
    \item $\int_0^\infty (\widetilde {\eta}_{k+1}-\widetilde {\eta}_{k})d\gamma_\alpha=\int_0^\infty \widetilde {\eta}_{k+1}d\gamma_\alpha-\int_0^\infty \widetilde {\eta}_{k}d\gamma_\alpha=0$ for each $k\in \mathbb{N}$;
    \item $\supp\left(\widetilde {\eta}_{k+1}-\widetilde {\eta}_{k}\right)\subset \left[\sqrt{k-1},\sqrt{k+2}\right]$, for $k\in \mathbb{N}\setminus\{0\}$, while $\widetilde {\eta}_1-\widetilde {\eta}_{0}$ is supported on $\left[0,\sqrt{2}\right]$;
    \item When $k\in\mathbb{N}\setminus\{0\}$, $\left[\sqrt{k-1},\sqrt{k+2}\right]\subset D_k$, where
\[
D_k:= \left[\tfrac{\sqrt{k-1}+\sqrt{k+2}}{2}-\tfrac{2}{\sqrt{k-1}+\sqrt{k+2}},\tfrac{\sqrt{k-1}+\sqrt{k+2}}{2}+\tfrac{2}{\sqrt{k-1}+\sqrt{k+2}}\right],
\]
and
\[\|\widetilde {\eta}_{k+1}-\widetilde {\eta}_{k}\|_\infty\le \frac{1}{\gamma_\alpha(\frac{1}{2}I_k)}+\frac{1}{\gamma_\alpha(\frac{1}{2}I_{k+1})}\le C\frac{1}{\gamma_\alpha(D_k)}.\]
While
\[\|\widetilde {\eta}_1-\widetilde {\eta}_{0}\|_\infty \le \frac{1}{\gamma_\alpha\left(0,\sqrt{2}\right)}+\frac{1}{\gamma_\alpha(0,\frac12)}\le C\frac{1}{\gamma_\alpha\left(0,\sqrt{2}\right)}.\]
\end{enumerate}

Now, we are going to see that
\[
\sum_{k=0}^\infty \left|\int_0^\infty f(y)\mu_k(y)d\gamma_\alpha(y)\right|\le C\left(\|f\|_{L^1((0,\infty),\gamma_\alpha)}+\mathcal{E}_\alpha(f)\right).
\]
We can write, for every $k\in \mathbb{N}$,
\begin{align*}
\int_0^\infty f(x)\mu_k(x)d\gamma_\alpha(x)&=\int_0^\infty\left(\int_0^x\mu_k'(y)dy+\mu_k\left(0^+\right)\right)f(x)d\gamma_\alpha(x)\\
&=\int_0^\infty\mu_k'(y)\int_y^\infty f(x)d\gamma_\alpha(x)dy+\mu_k\left(0^+\right)\int_0^\infty f(x)d\gamma_\alpha(x).
\end{align*}
Here $\phi\left(0^+\right)$ stands for $\lim_{x\rightarrow 0^+}\phi(x)$. We have that $\mu_k\left(0^+\right)=\sum_{j=k}^\infty\eta_j\left(0^+\right)=0$, for each $k\in \mathbb{N}$, $k\ge 1$, and $\mu_0\left(0^+\right)=1$. As in \cite[p.~1687]{MMS-max} we get that $\supp(\mu_k')\subset I_k$ and 
\[
|\mu_k'(y)|\le \frac{C}{r_k}\le C(1+c_k),\quad k\in \mathbb{N},
\]
where $C$ is independent of $k$. Hence,
\begin{align*}
\sum_{k=0}^\infty \left|\int_0^\infty f(y)\mu_k(y)d\gamma_\alpha(y)\right|&\le C\sum_{k=0}^\infty\int_{I_k}(1+c_k
)\left|\int_y^\infty f(x)d\gamma_\alpha(x)\right|dy\\
&\quad +\|f\|_{L^1((0,\infty),\gamma_\alpha)}\\
&\le C\int_0^\infty (1+y)\left|\int_y^\infty f(x)d\gamma_\alpha(x)\right|dy\\
&\quad +\|f\|_{L^1((0,\infty),\gamma_\alpha)}\\
&\le C\left(\mathcal{E}_\alpha(f)+\|f\|_{L^1((0,\infty),\gamma_\alpha)}\right).
\end{align*}
Our purpose is established and the proof of the ``if'' implication of Theorem \ref{H1-Malfaloc} is finished.

We will now prove the ``only if'' implication of Theorem \ref{H1-Malfaloc}. We will see that, for every $f\in \mathcal{H}^1((0,\infty),\gamma_\alpha)$,
\begin{equation}\label{estH1}
\|\mathcal{M}_{\alpha,\loc}(f)\|_{L^1((0,\infty),\gamma_\alpha)}+\mathcal{E}_\alpha(f)\le C\|f\|_{\mathcal{H}^1((0,\infty),\gamma_\alpha)}.
\end{equation}
In order to do this we will first show that there exists $C>0$ such that, for every~$(1,\infty,\alpha)$-atom $b$,
\begin{equation}\label{estato}
\|\mathcal{M}_{\alpha,\loc}(b)\|_{L^1((0,\infty),\gamma_\alpha)}+\mathcal{E}_\alpha(b)\le C.
\end{equation}
Assume that $b$ is a $(1,\infty,\alpha)$-atom associated with the interval $I=I(x_0,r_0)$ where $0<r_0\le \min\{x_0,m(x_0)\}$. According to Lemma \ref{L1}, $\mathcal{M}_{\alpha,\loc}(b)$ is supported on the interval 
$\widetilde{I}=I(x_0, hm(x_0))$. Let $x\in (0,\infty)$ and suppose that $\phi\in \mathcal{A}_{x,\loc}^\alpha$ is associated to the interval $J=I(x,r_1)$. Then,
\begin{align*}
\left|\int_0^\infty b(y)\phi(y)d\mathfrak{m}_\alpha(y)\right|&\le \|b\|_{L^\infty((0,\infty),\gamma_\alpha)}\int_{I\cap J}|\phi(y)|d\mathfrak{m}_\alpha(y)\\
&\le \|b\|_{L^\infty((0,\infty),\gamma_\alpha)}\frac{\mathfrak{m}_\alpha(I\cap J)}{\mathfrak{m}_\alpha(J)}\\
&\le \frac{1}{\gamma_\alpha(I)},
\end{align*}
which yields
\[
\mathcal{M}_{\alpha,\loc}(b)(x)\le \frac{1}{\gamma_\alpha(I)}.
\]
We deduce that
\[
\int_{2I}\mathcal{M}_{\alpha,\loc}(b)(x)d\gamma_\alpha(x)\le\frac{\gamma_\alpha(2I)}{\gamma_\alpha(I)}\le C.
\]

Suppose now that $x\in \widetilde{I}\setminus (2I)$ and $\phi\in \mathcal{A}_{x,\loc}^\alpha$ associated to the interval $J$ given above. We can write
\begin{align*}
\int_0^\infty b(y)\phi(y)d\mathfrak{m}_\alpha(y)&=\int_0^\infty b(y)(\phi(y)-\phi(x_0))d\mathfrak{m}_\alpha(y)+\phi(x_0)\int_0^\infty b(y)d\mathfrak{m}_\alpha(y), 
\end{align*}
and estimate each term separately.

Since $\supp(\phi)\subset J$ and $\supp(b)\subset I$ we can assume that $r_1>\dist (x,I)$ (otherwise, if $0<r_1\le \dist (x,I)$, then $J\subset I^c$ and $\int_0^\infty b\phi d\mathfrak{m}_\alpha=0$). Since $\phi\in \mathcal{A}_{x,\loc}^\alpha$, by \cite[(1.4)]{YY} we have that
\begin{align*}
\left|\int_0^\infty b(y)(\phi(y)-\phi(x_0))d\mathfrak{m}_\alpha(y)\right|&\le \int_0^\infty |b(y)||\phi(y)-\phi(x_0)|d\mathfrak{m}_\alpha(y)\\
&\le \frac{1}{r_1\mathfrak{m}_\alpha(J)}\int_I|y-x_0||b(y)|d\mathfrak{m}_\alpha(y)\\
&\le C\frac{r_0}{r_1\mathfrak{m}_\alpha(J)}\int_I|b(y)|\frac{\gamma(y)}{\gamma(x_0)}d\mathfrak{m}_\alpha(y)\\
&\le C\frac{r_0}{r_1\mathfrak{m}_\alpha(J)\gamma(x_0)}\\
&\le C\frac{r_0}{\dist (x,I)\mathfrak{m}_\alpha(I(x, \dist(x,I)))\gamma(x_0)}\\
&\le C\frac{r_0}{d(x,I)^2x^{2\alpha+1}\gamma(x_0)}.
\end{align*}

On the other hand, since $\int_0^\infty b(y)d\gamma_\alpha(y)=0$, we get
\[
\int_0^\infty b(y)d\mathfrak{m}_\alpha(y)=\int_0^\infty b(y)y^{2\alpha+1}\frac{\gamma(x_0)-\gamma(y)}{\gamma(x_0)}dy.
\]
As in \cite[(3.6)]{MM} we obtain 
\[
\left|\frac{\gamma(x_0)-\gamma(y)}{\gamma(x_0)}\right|\le Cr_0(1+x_0),\quad y\in I.
\]
Then, using that $\int_0^\infty|b(y)|d\gamma_\alpha(y)\le C$,  we get
\[
\left|\int_0^\infty b(y)d\mathfrak{m}_\alpha(y)\right|\le C\frac{r_0(1+x_0)}{\gamma(x_0)}.
\]
By applying again \cite[(3.6)]{YY}, we have that
\begin{align*}
\left|\phi(x_0)\int_0^\infty b(y)d\mathfrak{m}_\alpha(y)\right|&\le C |\phi(x_0)|\frac{r_0(1+x_0)}{\gamma(x_0)}\\
&\le C\frac{r_0(1+x_0)}{\gamma(x_0)\mathfrak{m}_\alpha(J)}\\
&\le C\frac{r_0(1+x_0)}{\gamma(x_0)x^{2\alpha+1}r_1}\\
&\le C\frac{r_0(1+x_0)}{\gamma(x_0)x^{2\alpha+1}\dist (x,I)}.
\end{align*}
By combining the above estimates we conclude that, for every $x\in \widetilde{I}\setminus (2I)$,
\[
\mathcal{M}_{\alpha,\loc}(b)(x)\le C\left(\frac{r_0}{d(x,I)^2x^{2\alpha+1}\gamma(x_0)}+\frac{r_0(1+x_0)}{\dist (x,I)x^{2\alpha+1}\gamma(x_0)}\right).
\]
We deduce that
\begin{align*}
\int_{\widetilde{I}\setminus (2I)}&\mathcal{M}_{\alpha,\loc}(b)(x)d\gamma_\alpha(x)\\
&\le C\left(\int_{\widetilde{I}\setminus (2I)}\frac{r_0}{d(x,I)^2\gamma(x_0)}d\gamma(x)+\int_{\widetilde{I}\setminus (2I)}\frac{r_0(1+x_0)}{\dist (x,I)\gamma(x_0)}d\gamma(x)\right)\\
&\le C\left(r_0\int_{(2I)^c}\frac{dx}{\dist (x,I)^2}+r_0(1+x_0)\int_{\widetilde{I}\setminus (2I)}\frac{dx}{\dist (x,I)}\right)\\
&\le C\left(1+r_0(1+x_0)\log\left(\frac{m(x_0)}{r_0}\right)\right)\\
&\le C\left(1+\frac{r_0}{m(x_0)}\log\left(\frac{m(x_0)}{r_0}\right)\right)\le C.
\end{align*}
It follows that
\[
\int_0^\infty \mathcal{M}_{\alpha,\loc}(b)(x)d\gamma_\alpha(x)\le C,
\]
where $C>0$ does not depend on $b$.

Clearly, if $b(x)=1$ for every $x\in (0,\infty)$, then $\mathcal{M}_{\alpha,\loc}(b)\le 1$ and
\[
\int_0^\infty \mathcal{M}_{\alpha,\loc}(b)(x)d\gamma_\alpha(x)\le 1
\]
for this atom.

\refstepcounter{BDQSc}\label{proof-Ealfa} We now prove that, for a certain $C>0$, $\mathcal{E}_\alpha(b)\le C$, for every $(1,\infty,\alpha)$-atom $b$. Suppose that $b$ is a $(1,\infty,\alpha)$-atom associated to the interval $I=I(x_0,r_0)$. Since $\int_0^\infty bd\gamma_\alpha=0$ we deduce that $\int_x^\infty bd\gamma_\alpha=0$ for $x\notin I$, and therefore
\[
\left|\int_x^\infty b(y)d\gamma_\alpha(y)\right|\le \int_I|b(y)|d\gamma_\alpha(y)\chi_I(x)\le \chi_I(x),\quad x\in (0,\infty).
\]
Hence,
\begin{align*}
\mathcal{E}_\alpha(b)&\le\int _Ix\left|\int_x^\infty b(y)d\gamma_\alpha(y)\right|dx\\
&\le \int_I|b(y)|\int_{x_0-r_0}^yxdxd\gamma_\alpha(y)\\
&\le \int_I|b(y)|d\gamma_\alpha(y)\left(\frac{(x_0+r_0)^2-(x_0-r_0)^2}{2}\right)\le 2.
\end{align*}
In the last inequality we have taken into account that $x_0r_0\le 1$.

On the other hand, if $b(x)=1$ for every $x\in (0,\infty)$, we can write
\[
\mathcal{E}_\alpha(1)=\int_0^\infty x \left|\int_x^\infty d\gamma_\alpha(y)\right|dx\le \int_0^\infty\int_0^yxdxd\gamma_\alpha(y)\le C\int_0^\infty y^{2\alpha+3}e^{-y^2}dy\le C.
\]
Thus \eqref{estato} is proved.

As it was established by Bownik (\cite{Bo}) in the Euclidean setting, the property \eqref{estato} is not sufficient to deduce the validity of \eqref{estH1}. 

In order to obtain \eqref{estH1}, we are going to see that $\mathcal{M}_{\alpha,\loc}$ is bounded from $L^1((0,\infty),\gamma_\alpha)$ into $L^{1,\infty}((0,\infty),\gamma_\alpha)$. Let $f\in L^1((0,\infty),\gamma_\alpha)$ and $\lambda>0$. We consider the sequence $\{I_i\}_{i\in \mathbb{N}}$ introduced in the first part of this proof. We have that
\[
\{x\in (0,\infty):\mathcal{M}_{\alpha,\loc}(f)(x)>\lambda\}\subset \bigcup_{i=0}^\infty\{x\in I_i:\mathcal{M}_{\alpha,\loc}(f)(x)>\lambda\}.
\]
Then,
\begin{align*}
\gamma_\alpha(\{x\in (0,\infty):\mathcal{M}_{\alpha,\loc}(f)(x)>\lambda\})&\le \sum_{i=0}^\infty\gamma_\alpha(\{x\in I_i:\mathcal{M}_{\alpha,\loc}(f)(x)>\lambda\})\\
&\le \sum_{i=0}^\infty\gamma(c_i)\mathfrak{m}_\alpha(\{x\in I_i:\mathcal{M}_{\alpha,\loc}(f)(x)>\lambda\}).
\end{align*}
We can write, for every $x\in (0,\infty)$ and $\phi\in \mathcal{A}_{x,\loc}^\alpha$,
\[
\left|\int_0^\infty f(y)\phi(y)d\mathfrak{m}_\alpha(y)\right|\le \frac{1}{\mathfrak{m}_\alpha(I(x,r))}\int_{x-r}^{x+r}|f(y)|d\mathfrak{m}_\alpha(y),
\]
for some $0<r\le \min\{x,m(x)\}$. Then, since $m(c_i)\sim m(x)$ for every $x\in I_i$ and $ y\in I(x,r)$, we get $y\in CI_i$ and
\[
\int_{x-r}^{x+r}|f(y)|d\mathfrak{m}_\alpha(y)=\int_{x-r}^{x+r}|f(y)|\chi_{CI_i}(y)d\mathfrak{m}_\alpha(y).
\]
We deduce that, for $i\in \mathbb{N}$,
\[
\mathcal{M}_{\alpha,\loc}(f)(x)\le M_{\mathfrak{m}_\alpha}(f\chi_{CI_i})(x),\quad x\in I_i,
\]
where $M_{\mathfrak{m}_\alpha}$ denotes the centered Hardy-Littlewood  maximal function defined by the measure $\mathfrak{m}_\alpha$ on $(0,\infty)$.

It is well-known that $M_{\mathfrak{m}_\alpha}$ is bounded from $L^1((0,\infty),\mathfrak{m}_\alpha)$ into $L^{1,\infty}((0,\infty),\mathfrak{m}_\alpha)$. Thus,
\begin{align*}
\gamma_\alpha(\{x\in (0,\infty)&:\mathcal{M}_{\alpha,\loc}(f)(x)>\lambda\})
\\&\le \sum_{i=0}^\infty\gamma(c_i)\mathfrak{m}_\alpha(\{x\in I_i:M_{\mathfrak{m}_\alpha}(f\chi_{CI_i})(x)>\lambda\})\\
&\le \frac{C}{\lambda}\sum_{i=0}^\infty \gamma(c_i)\int_{CI_i}|f(y)|d\mathfrak{m}_\alpha(y)\\
&\le \frac{C}{\lambda}\sum_{i=0}^\infty\int_{CI_i}|f(y)|d\gamma_\alpha(y)\\
&\le \frac{C}{\lambda}\int_0^\infty |f(y)|d\gamma_\alpha(y).
\end{align*}
Therefore, we have proved that the local maximal function $\mathcal{M}_{\alpha,\loc}$ is bounded from~$L^1((0,\infty),\gamma_\alpha)$ into $L^{1,\infty}((0,\infty),\gamma_\alpha)$.

Suppose that $f=\sum_{j=0}^\infty\lambda_jb_j$, where, for every $j\in \mathbb{N}$, $b_j$ is a $(1,\infty,\alpha)$-atom and $\lambda_j\in \mathbb{C}$ being $\sum_{j=0}^\infty|\lambda_j|<\infty$. The series defining $f$ converges in $L^1((0,\infty),\gamma_\alpha)$ and from the boundedness proved above for $\mathcal{M}_{\alpha,\loc}$, 
\[
\mathcal{M}_{\alpha,\loc}(f)=\lim_{k\to\infty}\mathcal{M}_{\alpha,\loc}\left(\sum_{j=0}^k\lambda_jb_j\right),\quad \text{ in } L^{1,\infty}((0,\infty),\gamma_\alpha).
\]
Then, there exists an increasing function $\psi:\mathbb{N}\rightarrow \mathbb{N}$ such that
\[
\mathcal{M}_{\alpha,\loc}(f)(x)=\lim_{k\to\infty}\mathcal{M}_{\alpha,\loc}\left(\sum_{j=0}^{\psi(k)}\lambda_jb_j\right)(x),\quad \text{a.e. } x\in (0,\infty),
\]
which yields
\begin{align*}
\mathcal{M}_{\alpha,\loc}(f)(x)
&\le\lim_{k\to\infty}\sum_{j=0}^{\psi(k)}|\lambda_j|\mathcal{M}_{\alpha,\loc}(b_j)(x)\\
&=\sum_{j=0}^{\infty}|\lambda_j|\mathcal{M}_{\alpha,\loc}(b_j)(x),\quad \text{a.e. } x\in (0,\infty).
\end{align*}
By \eqref{estato} it follows that
\[
\|\mathcal{M}_{\alpha,\loc}(f)\|_{L^1((0,\infty),\gamma_\alpha)}\le \sum_{j=0}^\infty|\lambda_j|\|\mathcal{M}_{\alpha,\loc}(b_j)\|_{L^1((0,\infty),\gamma_\alpha)}\le C\sum_{j=0}^\infty |\lambda_j|,
\]
so we can conclude that $\|\mathcal{M}_{\alpha,\loc}(f)\|_{L^1((0,\infty),\gamma_\alpha)}\le C\|f\|_{\mathcal{H}^1((0,\infty),\gamma_\alpha)}$.

Since $f=\sum_{j=0}^\infty \lambda_jb_j$ in $L^1((0,\infty),\gamma_\alpha)$ we obtain
\[
\int_x^\infty f(y)d\gamma_\alpha(y)=\sum_{j=0}^\infty\lambda_j\int_x^\infty b_j(y)d\gamma_\alpha(y),\quad x\in (0,\infty).
\]
Then,
\[
\left|\int_x^\infty f(y)d\gamma_\alpha(y)\right|\le\sum_{j=0}^\infty|\lambda_j|\left|\int_x^\infty b_j(y)d\gamma_\alpha(y)\right|,\quad x\in (0,\infty).
\]
By using the monotone convergence theorem  and \eqref{estato} we get
\[
\mathcal{E}_\alpha(f)\le \sum_{j=0}^\infty |\lambda_j|\mathcal{E}_\alpha(b_j)\le C\sum_{j=0}^\infty|\lambda_j|\leq C\|f\|_{\mathcal{H}^1((0,\infty),\gamma_\alpha)}.
\]
The proof is now finished.

\section{Proof of Theorem \ref{equiv-atomicH1}}\label{sec-equiv-atomiH1}

\ref{mathbbH1} $\Rightarrow$ \ref{mathcalH1} It is sufficient to note that if $b$ is an~$(a,q,\alpha)_w$-atom, then $b$ is also an $(a,q,\alpha)$-atom, when $0<a\le 1$.

\ref{mathcalH1} $\Rightarrow$ \ref{maxlocconv} 
Let $f\in \mathcal{H}^1((0,\infty),\gamma_\alpha)$. Then $\mathcal{M}_{\alpha,\loc}(f)\in L^1((0,\infty),\gamma_\alpha)$ by Theorem~\ref{H1-Malfaloc}. We shall prove that $\mathbb{M}_{1,\loc}(f)\le C\mathcal{M}_{\alpha,\loc}(f)$.


Let $\phi\in \mathbb{A}$ and $t,x\in (0,\infty)$ such that $t\le w(x)$. We define \[\varphi_{x,t}(y)=\frac{1}{ty^{2\alpha+1}}\phi\left(\frac{x-y}{t}\right), y\in (0,\infty).\] We have that $\supp(\varphi_{x,t}) \subset{\left(x-t,x+t\right)}\subset \left(\frac{19}{20}x,\frac{21}{20}x\right)$ {and $I(x,t)\in \mathcal{B}$.} 
Then, by \cite[(1.4)]{YY}, for every $y\in (0,\infty)$,
\[
|\varphi_{x,t}(y)|\le \frac{1}{ty^{2\alpha+1}}\le \frac{C}{tx^{2\alpha+1}}\le \frac{C}{t(x+t)^{2\alpha+1}}\le \frac{C}{\mathfrak{m}_\alpha((I(x,t))}.
\]
On the other hand, we get
\[
\partial_y\varphi_{x,t}(y)=-\frac{2\alpha+1}{ty^{2\alpha+2}}\phi\left(\frac{x-y}{t}\right)-\frac{1}{t^2y^{2\alpha+1}}\phi'\left(\frac{x-y}{t}\right), \quad y\in (0,\infty).
\]
Then, for each $y\in (0,\infty)$,
\[
|\partial_y \varphi_{x,t}(y)|\le C\left(\frac{1}{ty^{2\alpha+2}}+\frac{1}{t^2y^{2\alpha+1}}\right)\le \frac{C}{t^2x^{2\alpha+1}}\le \frac{C}{t\mathfrak{m}_\alpha(I(x,t))},
\]
so it follows that, for a certain $C>0$, ${C\varphi_{x,t}}\in \mathcal{A}_{x,\loc}^\alpha$ {and, therefore, $\mathbb{M}_{1,\loc}(f)\le C\mathcal{M}_{\alpha,\loc}(f)$
as claimed.

Moreover, since $\mathbb{M}_{a,\loc}(f)\leq \mathbb{M}_{1,\loc}(f)$ for any $0<a\leq 1$, we also have that }$\mathbb{M}_{{a},\loc}(f)\in L^1((0,\infty),\gamma_\alpha)$ {for any $0<a\leq 1$} provided that $f\in \mathcal{H}^1((0,\infty),\gamma_\alpha)$.

By proceeding as in the proof of $\mathcal{E}_\alpha(f)<\infty$ for $f\in \mathcal{H}^1((0,\infty),\gamma_\alpha)$ (see Section~\ref{sec-H1-Malfaloc}, p.~\pageref{proof-Ealfa}), in order to see that $\mathbb{E}_\alpha(f)<\infty$ for such $f$, it is sufficient to prove that there exists $C> 0$ such that, for every $(1,\infty,\alpha)$-atom $b$,
\[
\mathbb{E}_{\alpha,1}(b):=\int_0^1\frac{1}{y}\left|\int_0^yb(x)d\gamma_\alpha(x)\right|dy\le C.
\]
If $b(x)=1$ for every $x\in (0,\infty)$, then
\begin{align*}
\mathbb{E}_{\alpha,1}(b)&=\int_0^1\frac{1}{y}\int_0^y x^{2\alpha+1}e^{-x^2}dxdy\\
&\le \int_0^1\frac{1}{y}\int_0^y x^{2\alpha+1}dxdy\\
&\le \frac{1}{2\alpha+2}\int_0^1y^{2\alpha+1}dy=(2\alpha+2)^{-2}.
\end{align*}
Suppose now that $b$ is a $(1,\infty,\alpha)$-atom associated to an interval $I(x_0,r_0)$ with $0<r_0\le x_0$ and $r_0\le m(x_0)$.

Since $\int_0^\infty b(y)d\gamma_\alpha(y)=0$, we have $\int_0^yb(x)d\gamma_\alpha(x)=0$, provided that $|y-x_0|\ge r_0$. Then,
\begin{align*}
\mathbb{E}_{\alpha,1}(b)&=\int_{(x_0-r_0,x_0+r_0)\cap (0,1)}\frac{1}{y}\left|\int_0^yb(x)d\gamma_\alpha(x)\right|dy\\
&\le C\frac{\gamma(x_0)}{\gamma_\alpha((x_0-r_0,x_0+r_0))}\int_{x_0-r_0}^{x_0+r_0}\frac{1}{y}\int_0^yx^{2\alpha+1}dxdy\\
&\le C\frac{\gamma(x_0)}{\gamma_\alpha((x_0-r_0,x_0+r_0))}\int_{x_0-r_0}^{x_0+r_0}y^{2\alpha+1}dy\\
&\le \frac{C}{\gamma_\alpha((x_0-r_0,x_0+r_0))}\int_{x_0-r_0}^{x_0+r_0}d\gamma_\alpha(y)\le C.
\end{align*}
Thus the proof of \ref{mathcalH1} $\Rightarrow$ \ref{maxlocconv} is finished.

\ref{maxlocconv} $\Rightarrow$ \ref{mathbbH1}  Let $a\le 1$. We will see that if $f\in L^1((0,\infty),\gamma_\alpha)$, then $f\in \mathbb{H}^{1,\infty}_{2a}((0,\infty),\gamma_\alpha)$ provided that \ref{maxlocconv} holds for this $a$. In order to prove this property we proceed as in the proof of the corresponding property in Theorem \ref{H1-Malfaloc} (see also the proof of \cite[Theorem~3.3]{MMS-max}). 
We sketch the main steps.

By using \cite[Lemma~2.1]{ChL1} (see also \cite[p.~276]{Dz}) we define the sequence of positive numbers $\{c_j\}_{j\in \mathbb{Z}}$ given by $c_0=1$, $c_j=c_{j-1}+aw(c_{j-1})$ for $j>0$ and $c_j=c_{j+1}-aw(c_{j+1})$ for $j<0$. Then, for every $j\in \mathbb{Z}$, the interval $\mathbb{I}_j=I(c_j,r_j)$ where $r_j=aw(c_j)$, verify the following properties
\begin{enumerate}
    \item $(0,\infty)=\bigcup_{j\in\mathbb{Z}}\mathbb{I}_j$;
    \item For every $k\in \mathbb{Z}$, $\mathbb{I}_k\cap \mathbb{I}_j=\emptyset$ provided that $j\notin\{k-1,k,k+1\}$;
\end{enumerate}
We choose a partition of unity $\{\eta_j\}_{j\in \mathbb{Z}}$ such that $|\eta_j'|\le C\frac{1}{r_
j}$, for every~$j\in \mathbb{Z}$.

Assume that $0<\beta<8$ and $|x-y|\le \beta w(x)$, with $x,y\in (0,\infty)$. We shall see that $w(x)/h(\beta)\le w(y)\le h(\beta) w(x)$ where $h$ is a positive, increasing and continuous function on $(0,8)$. An explicit form for the function $h$ can be obtained. We first consider $x\ge 1$, so $w(x)=\frac{1}{8x}$. Then, $x-\frac{\beta}{8x}\le y\le x+\frac{\beta}{8x}$. If $1\le y\le x+\frac{\beta}{8x}=\frac{8x^2+\beta}{8x}$, since $y>x-\beta w(x)=\frac{8x^2-\beta}{8x}$, we get 
\begin{align*}
 w(y)&=\frac{1}{8y} \le \frac{x}{8x^2-\beta}  \le \frac{x}{(8+\beta)x^2}=\frac{1}{(8-\beta)x}=\frac{8}{8-\beta}w(x).
\end{align*}
Also we have that
\begin{align*}
   w(y)&= \frac{1}{8y} \ge \frac{x}{8x^2+\beta} \ge \frac{x}{(8+\beta)x^2}=\frac{1}{(8+\beta)x}=\frac{8}{8+\beta}w(x).
\end{align*}
Suppose now that $x-\beta w(x)<y<1$. Then, $w(y)=\frac{y}{8}$. Since $x-\frac{\beta}{8x}<1$, it follows that $x<z:=\frac{2+\sqrt{4+2\beta}}{4}$. We obtain
$$
w(y)<\frac{1}{8}<\frac{z}{8x}=zw(x).
$$
Since 
$$
\frac{8-\beta}{8}=1-\frac{\beta}{8}\le x-\frac{\beta}{8x}\le y,
$$
we get
$$
w(x)\le \frac{1}{8}\le \frac{1}{8-\beta}y\le \frac{8}{8-\beta}w(y).
$$

Assume now that $x<1$. We can proceed in a similar way. We have that $w(x)=\frac{x}{8}$ and $(1-\frac{\beta}{8})x<y<(1+\frac{\beta}{8})x$. If $y<1$, then
$$
\frac{8-\beta}{8}w(x)<w(y)=\frac{y}{8}<\frac{8+\beta}{8}\frac{x}{8}=\frac{8+\beta}{8}w(x).
$$
Suppose that $y\ge 1$. Since $x>\frac{8}{8+\beta}y$, we obtain
$$
w(x)=\frac{x}{8}=\frac{x^2}{8x}\le \frac{1}{8x}<\frac{8+\beta}{8}\,\frac{1}{8y}=\frac{8+\beta}{8}w(y),
$$
and
$$
w(y)=\frac{1}{8y}\le \frac{x}{8x}<\frac{8+\beta}{8}\,\frac{x}{8}=\frac{8+\beta}{8}w(x).
$$


Suppose now that $g\in L^1_{\loc}((0,\infty),\gamma_\alpha)$ and $j\in \mathbb{Z}$. Let $\phi\in \mathbb{A}$, $x\in (0,\infty)$ and $0<t\le aw(x)$. Then, since $|x-c_j|\le |x-y|+|y-c_j|\le a(w(x)+w(c_j)\le 3aw(c_j)$ provided that $y\in \mathbb{I}_j$ and $|x-y|<t$, it follows that
\[
(\phi_t*(g\eta_j\gamma_\alpha))(x)=0,\quad x\notin 3\mathbb{I}_j.
\]
Therefore $\supp(\mathbb{M}_{a,\loc}(g\eta_j\gamma_\alpha))\subset 3\mathbb{I}_j$.

We shall now see that
\begin{equation}\label{3.2}
\mathbb{M}_{a,\loc}(g\eta_j\gamma_\alpha)(x)\le C\mathbb{M}_{a,\loc}(g)(x)\gamma_\alpha(c_j)\chi_{3\mathbb{I}_j}(x),\quad x\in (0,\infty),
\end{equation}
by showing that $C\varphi_{x,t}\in \mathbb{A}$, for every $x\in 3\mathbb{I}_j$ and $0<t\le aw(x)$ and some constant $C>0$ that does not depend on $x$ and $t$, being
\[
\varphi_{x,t}(z)=\phi(z)\eta_j(x-tz)\gamma_\alpha(x-tz)/\gamma_\alpha(c_j),\quad z\in (0,\infty).
\]

Let $x\in 3\mathbb{I}_j$ and $0<t\le aw(x)$. Since there exists $C>0$ such that
\begin{equation}\label{3.1}
\frac{1}{C}\gamma_\alpha(y)\le \gamma_\alpha(c_j)\le C\gamma_\alpha(y),\quad y\in 5\mathbb{I}_j,
\end{equation}
and
\[
|x-tz-c_j|\le |x-c_j|+t|z|\le 3aw(c_j)+aw(x)\le 5aw(c_j), \quad |z|\le 1,
\]
we obtain that $|\varphi_{x,t}(z)|\le C$ for any $|z|\le 1$.

On the other hand, 
\begin{align*}
\partial_z\varphi_{x,t}(z)=\frac{1}{\gamma_\alpha(c_j)} & \left(\phi'(z)\eta_j(x-tz)\gamma_\alpha(x-tz)-t\phi(z)\eta_j'(x-tz)\gamma_\alpha(x-tz)\right. \\
&\left. \quad -t\phi(z)\eta_j(x-tz)\gamma_\alpha'(x-tz)\right),\quad |z|\le 1,
\end{align*}
which leads to
\begin{align*}
|\partial_z\varphi_{x,t}(z)|&\le \frac{C}{\gamma_\alpha(c_j)}\left(\gamma_\alpha(x-tz)+\frac{aw(x)\gamma_\alpha(x-tz)}{r_j}\right.\\
&\left.\quad +aw(x)\gamma_\alpha(x-tz)\left(\frac{1}{x}+x\right)\right)\le C,\quad |z|\le 1.
\end{align*}
Thus, for a certain $C>0$, $C\varphi_{x,t}\in \mathbb{A}$ and
 we can write, for every $x\in (0,\infty)$  and $0<t\le aw(x)$,
\[
(\phi_t*(g\eta_j\gamma_\alpha))(x)=((\varphi_{x,t})_t*g)(x)\gamma_\alpha(c_j)\chi_{3\mathbb{I}_j}(x),
\]
so~\eqref{3.2} follows.

Let $f\in L^1((0,\infty),\gamma_\alpha)$ such that $\int_0^\infty fd\gamma_\alpha=0$ and $j\in \mathbb{Z}$. We define
\begin{equation}\label{3.3}
b_j=\frac{\int_0^\infty f\eta_jd\gamma_\alpha}{\int_0^\infty \eta_jd\gamma_\alpha}.
\end{equation}

 Our next objective is to see that $(f-b_j)\eta_j\gamma_\alpha\in H^1(\mathbb{R},dx)$ and
\[
\|(f-b_j)\eta_j\gamma_\alpha\|_{H^1(\mathbb{R},dx)}\le C\int_{3\mathbb{I}_j}\mathbb{M}_{a,\loc}(f)(x)d\gamma_\alpha(x).
\]
Here $H^1(\mathbb{R},dx)$ denotes the classical Hardy space on $\mathbb{R}$. 

Let $\phi\in \mathbb{A}\cap C_c^\infty(0,\infty)$ such that $\int_\mathbb{R}\phi(x) dx\neq 0$. By using \eqref{3.2} and \eqref{3.3} we have that
\begin{align*}
\int_0^\infty \sup_{0<t\le aw(x)} &|\phi_t*((f-b_j)\eta_j\gamma_\alpha)(x)|dx\\
&\le C\gamma_\alpha(c_j)\int_{3\mathbb{I}_j}\mathbb{M}_{a,\loc}(f-b_j)(x)dx\\
&\le C\gamma_\alpha(c_j)\left(\int_{3\mathbb{I}_j}\mathbb{M}_{a,\loc}(f)(x)dx+|b_j|\int_{3\mathbb{I}_j}dx\right)\\
&\le C\left(\int_{3\mathbb{I}_j}\mathbb{M}_{a,\loc}(f)(x)d\gamma_\alpha(x)+\frac{\gamma_\alpha(3\mathbb{I}_j)}{\gamma_\alpha(\mathbb{I}_j)}\int_{\mathbb{I}_j}|f(x)|d\gamma_\alpha(x)\right)\\
&\le C\int_{3\mathbb{I}_j}\mathbb{M}_{a,\loc}(f)(x)d\gamma_\alpha(x),
\end{align*}
where we have used that, since $\eta_j(z)=1$ for $z\in \frac{1}{2}\mathbb{I}_j$,  $\int_0^\infty\eta_jd\gamma_\alpha\ge C\gamma_\alpha(\mathbb{I}_j)$.
We are going to see that
\[
\int_0^\infty \sup_{t> aw(x)}|\phi_t*((f-b_j)\eta_j\gamma_\alpha)(x)|dx\le C\int_{3\mathbb{I}_j}\mathbb{M}_{a,\loc}(f)(x)d\gamma_\alpha(x).
\]
We firstly note that for every $x\in 3\mathbb{I}_j$ and $t\ge aw(x)$, \eqref{3.3} leads to
\begin{align*}
|\phi_t*((f-b_j)\eta_j\gamma_\alpha)(x)|&\le \frac{C}{t}\int_{\mathbb{I}_j}|f(y)-b_j|d\gamma_\alpha(y)\\
&\le \frac{C}{w(x)}\int_{\mathbb{I}_j}(|f(y)|+|b_j|)d\gamma_\alpha(y)\\
&\le \frac{C}{w(c_j)}\int_{\mathbb{I}_j}|f(y)|d\gamma_\alpha(y).
\end{align*}
Then,
\begin{align*}
\int_{3\mathbb{I}_j}\sup_{t> aw(x)}|\phi_t*((f-b_j)\eta_j\gamma_\alpha)(x)|dx&\le C\frac{|\mathbb{I}_j|}{w(c_j)}\int_{\mathbb{I}_j}|f(y)|d\gamma_\alpha(y)\\
&\le C\int_{\mathbb{I}_j}\mathbb{M}_{a,\loc}(f)(y)d\gamma_\alpha(y).
\end{align*}
Here, $|A|$ denotes the Lebesgue measure of $A$, for every measurable set $A\subset \mathbb{R}$.

On the other hand, by proceeding as in the corresponding estimation in Section~\ref{sec-H1-Malfaloc} we can deduce that
\[
\int_{(3\mathbb{I}_j)^c}\sup_{t> aw(x)}|\phi_t*((f-b_j)\eta_j\gamma_\alpha)(x)|dx\le C\int_{\mathbb{I}_j}\mathbb{M}_{a,\loc}(f)(y)d\gamma_\alpha(y).
\]
Thus we prove that
\[
\int_0^\infty\sup_{t> aw(x)}|\phi_t*((f-b_j)\eta_j\gamma_\alpha)(x)|dx\le C\int_{\mathbb{I}_j}\mathbb{M}_{a,\loc}(f)(y)d\gamma_\alpha(y).
\]

Then,
\[
\|(f-b_j)\eta_j\gamma_\alpha\|_{H^1(\mathbb{R},dx)}\le C\int_{3\mathbb{I}_j}\mathbb{M}_{a,\loc}(f)(y)d\gamma_\alpha(y).
\]
By using now \cite[Lemma~2.1]{MMS-max} we can see that
\[
(f-b_j)\eta_j\gamma_\alpha=\sum_{i=0}^\infty \lambda_{i,j}g_{i,j}, \quad \text{ in }L^1((0,\infty),dx),
\]
where, for every $i\in \mathbb{N}$, $g_{i,j}$ is a $(1,\infty)$-classical atom supported on $2\mathbb{I}_j$ and $\lambda_{i,j}\in \mathbb{C}$ being
\[
\sum_{i=0}^\infty |\lambda_{i,j}|\le C \int_{3\mathbb{I}_j}\mathbb{M}_{a,\loc}(f)(y)d\gamma_\alpha(y).
\]
According to \eqref{3.1}, there exists $C>0$ that is not depending on $j$ such that $Cg_{i,j}/\gamma_\alpha$ is a $(2a,\infty,\alpha)_w$-atom, for every $i\in \mathbb{N}$, and
\[
(f-b_j)\eta_j=\sum_{i=0}^\infty \lambda_{ij}g_{i,j}/\gamma_\alpha, \quad \text{ in }L^1((0,\infty),\gamma_\alpha).
\]
As in Section~\ref{sec-H1-Malfaloc} we can prove that
\[
\sum_{j\in \mathbb{Z}}(f-b_j)\eta_j=\sum_{j\in \mathbb{Z}}\sum_{i=0}^\infty \lambda_{i,j}g_{i,j}/\gamma_\alpha, \quad \text{ in }L^1((0,\infty),\gamma_\alpha),
\]
and
\[
\sum_{j\in \mathbb{Z}}\sum_{i=0}^\infty |\lambda_{i,j}|\le C \int_0^\infty\mathbb{M}_{a,\loc}(f)(y)d\gamma_\alpha(y).
\]

As in \cite[p.~1687]{MMS-max} we define $\mu_k=\sum_{j=k}^\infty\eta_j$, $k\in \mathbb{Z}$, and we also consider $\widetilde {\eta}_j=\eta_j/\int_0^\infty\eta_jd\gamma_\alpha$, $j\in \mathbb{Z}$. By using summation by parts we have that
\[
\sum_{j\in \mathbb Z}b_j\widetilde {\eta}_j=\sum_{k=-\infty}^{+\infty}(\widetilde {\eta}_k-\widetilde {\eta}_{k-1})\int_0^\infty f\mu_kd\gamma_\alpha.
\]
There exists $C>0$ such that $C(\widetilde {\eta}_k-\widetilde {\eta}_{k-1})$ is a $(2a,\infty,\alpha)_w$-atom, for every $k\in \mathbb{Z}$. On the other hand recalling that $\int fd\gamma_\alpha = 0$, we get
\begin{align*}
\int_0^\infty f(x)\mu_k(x)d\gamma_\alpha(x)&=\int_0^1 f(x)\left(-\int_x^1\mu_k'(y)dy+\mu_k(1)\right)d\gamma_\alpha(x)\\
&\quad +\int_1^\infty f(x)\left(\int_1^x\mu_k'(y)dy+\mu_k(1)\right)d\gamma_\alpha(x)\\
&=-\int_0^1\mu_k'(y)\int_0^yf(x)d\gamma_\alpha(x)dy\\
&\quad +\int_1^\infty\mu_k'(y)\int_y^\infty f(x)d\gamma_\alpha(x)dy,\quad k\in \mathbb{Z}.
\end{align*}
Since, for every $k\in \mathbb{Z}$, $\supp(\mu_k')\subset \mathbb{I}_k$ and $|\mu_k'|\le C/w(c_k)$, we obtain
\begin{align*}
\sum_{k\in \mathbb{Z}}\left|\int_0^\infty f(x)\mu_k(x)d\gamma_\alpha(x)\right|&\le C\left(\sum_{k\le 0}\int_{\mathbb{I}_k}\frac{1}{c_k}\left|\int_0^yf(x)d\gamma_\alpha(x)\right|dy\right.\\
&\left.\quad +\sum_{k>0}\int_{\mathbb{I}_k}c_k\left|\int_y^\infty f(x)d\gamma_\alpha(x)\right|dy\right)\\
&\le C\left(\sum_{k\le 0}\int_{\mathbb{I}_k}\frac{1}{y}\left|\int_0^yf(x)d\gamma_\alpha(x)\right|dy\right.\\
&\left.\quad +\sum_{k>0}\int_{\mathbb{I}_k}y\left|\int_y^\infty f(x)d\gamma_\alpha(x)\right|dy\right)\\
&\le C\mathbb{E}_\alpha(f).
\end{align*}
We write $f=\sum_{j\in \mathbb Z}(f-b_j)\eta_j+\sum_{j\in \mathbb Z}b_j\eta_j$ and conclude that
$f\in \mathbb{H}_{2a}^{1,\infty}((0,\infty),\gamma_\alpha)$ with
\[
\|f\|_{\mathbb{H}_{2a}^{1,\infty}((0,\infty),\gamma_\alpha)}\le C\left(\int_0^\infty \mathbb{M}_{a,\loc}(f)(x)d\gamma_\alpha(x)+\mathbb{E}_\alpha(f)\right),
\]
where the constant $C>0$ does not depend on $f$.


\vspace{3mm}

The equivalences of the quantities $\|f\|_{\mathcal{H}^1((0,\infty),\gamma_\alpha)}$, $\|f\|_{\mathbb{H}^{1,q}_a((0,\infty),\gamma_\alpha)}$ and \[\|\mathbb{M}_{a,\loc}(f)\|_{ L^1((0,\infty),\gamma_\alpha)}+\mathbb{E}_\alpha(f)\] were established during the proofs of \ref{mathbbH1} $\Rightarrow$ \ref{mathcalH1} $\Rightarrow$ \ref{maxlocconv} $\Rightarrow$ \ref{mathbbH1}.

\section{Proof of Theorem \ref{ThH}.}

By using an integral representation for the modified Bessel function $I_\nu$ (\cite[(5.10.22)]{Leb} we deduce that, for every $t,x,y\in (0,\infty)$
\[
W_t^\alpha(x,y)=\frac{1}{(1-e^{-t})^{\alpha+1}}\int_{-1}^1 \exp\left(-\frac{q(e^{-t/2}x,y,s)}{1-e^{-t}}+y^2\right)\Pi_\alpha(s)ds,
\]
where $\Pi_\alpha(s)=\frac{\Gamma(\alpha+1)}{\Gamma\left(\alpha+\frac12\right)\sqrt{\pi}}(1-s^2)^{\alpha-1/2}$, $s\in (-1,1)$, and $q(x,y,s)=x^2+y^2-2xys$, $x,y\in (0,\infty)$ and $s\in (-1,1)$.

Let $\delta>0$. As in \cite{Sa4} we split $(0,\infty)\times (0,\infty)\times (-1,1)$ in two parts. The first part is named the $\delta$-local part and it is defined by
\[
L_\delta=\left\{(x,y,s)\in (0,\infty)\times (0,\infty)\times (-1,1): \sqrt{q(x,y,s)}\le\frac{\delta}{1+x+y}\right\}.
\]
The second part is $G_\delta=((0,\infty)\times (0,\infty)\times (-1,1))\setminus L_\delta$ and it is named the $\delta$-global part.

We now decompose, for every $t>0$, the operator $W_t^\alpha$ in the global and local parts as follows. We choose an smooth function $\psi$ defined in $(0,\infty)\times (0,\infty)\times (-1,1)$ such that $0\le\psi\le 1$, $\psi(x,y,s)=1$, $(x,y,s)\in L_1$, $\psi(x,y,s)=0$, $(x,y,s)\in G_2$, and
\[
|\partial_x\psi(x,y,s)| + |\partial_y\psi(x,y,s)|\le \frac{C}{\sqrt{q(x,y,s)}}, \quad (x,y,s)\in (0,\infty)\times (0,\infty)\times (-1,1).
\]
For every $t,x,y>0$, we define the local part $W_{t,\loc}^\alpha(x,y)$ of $W_t^\alpha(x,y)$ by
\[
W_{t,\loc}^\alpha(x,y)=\frac{1}{(1-e^{-t})^{\alpha+1}}\int_{-1}^1 \exp\left(-\frac{q(e^{-t/2}x,y,s)}{1-e^{-t}}+y^2\right)\psi(x,y,s)\Pi_\alpha(s)ds,
\]
and the global part $W_{t,\glob}^\alpha(x,y)$ of $W_t^\alpha(x,y)$ by
\[
W_{t,\glob}^\alpha(x,y)=W_t^\alpha(x,y)-W_{t,\loc}^\alpha(x,y).
\]
We define, for every $t>0$,
\[
W_{t,\loc}^\alpha(f)(x)=\int_0^\infty W_{t,\loc}^\alpha(x,y)f(y)d\gamma_\alpha(y),\quad x\in (0,\infty),
\]
and
\[
W_{t,\glob}^\alpha(f)(x)=\int_0^\infty W_{t,\glob}^\alpha(x,y)f(y)d\gamma_\alpha(y),\quad x\in (0,\infty).
\]
Finally we consider the local and global maximal operators defined by
\[
\mathbb{W}_{*,\loc}^\alpha(f)=\sup_{0<t<m(x)^2}|W_{t,\loc}^\alpha(f)|\quad \text{ and }\quad \,\mathbb{W}_{*,\glob}^\alpha(f)=\sup_{0<t<m(x)^2}|W_{t,\glob}^\alpha(f)|.
\]
It is clear that $\mathbb{W}_*^\alpha(f)\le \mathbb{W}_{*,\loc}^\alpha(f)+\mathbb{W}_{*,\glob}^\alpha(f)$.

We are going to prove first that the operator $\mathbb{W}_{*,\loc}^\alpha$ is bounded from the Hardy space $\mathcal{H}^1((0,\infty),\gamma_\alpha)$ into $L^1((0,\infty),\gamma_\alpha)$. Since $\mathbb{W}_{*}^\alpha$ is bounded from $L^1((0,\infty),\gamma_\alpha)$ into $L^{1,\infty}((0,\infty),\gamma_\alpha)$, $\mathbb{W}_{*,\loc}^\alpha$ has the same property. Then, in order to prove our objective it is sufficient to see that there exists $C>0$ such that
\begin{equation}\label{L1W*loc}
    \|\mathbb{W}_{*,\loc}^\alpha(b)\|_{L^1((0,\infty),\gamma_\alpha)}\le C,
\end{equation}
for every $(1,\infty,\alpha)$-atom $b$.

If $b(x)=1$, for every $x\in (0,\infty)$, then $\mathbb{W}_{*,\loc}^\alpha(b)(x)\le 1$, for any $x\in (0,\infty)$, and $\|\mathbb{W}_{*,\loc}^\alpha(b)\|_{L^1((0,\infty),\gamma_\alpha)}\le 1$.

On the other hand, suppose that $b$ is a $(1,\infty,\alpha)$-atom associated with an interval $I=(c_I-r_I,c_I+r_I)\in \mathcal{B}$ where $0<r_I\le c_I$. We define $a=b\gamma$. We can write
\begin{align*}
\|a\|_{L^\infty((0,\infty),\mathfrak{m}_\alpha)}&\le C\gamma(c_I)\|b\|_{L^\infty((0,\infty);\gamma_\alpha)}\\
&\le C\gamma(c_I)(\gamma_\alpha(I))^{-1}\\
&\le C\mathfrak{m}_\alpha(I),
\end{align*}
where $C>0$ does not depend on $b$. Then, $a/C$ is a $(\mathfrak{m}_\alpha,\infty)$-atom.

For every $t,x\in (0,\infty)$,
\begin{align*}
W_{t,\loc}^\alpha(b)(x)&=\int_0^\infty b(y)\int_{-1}^1\frac{\exp\left(-\frac{q(e^{-t/2}x,y,s)}{1-e^{-t}}+y^2\right)}{(1-e^{-t})^{\alpha+1}}\psi(x,y,s)\Pi_\alpha(s)dsd\gamma_\alpha(y)\\
&=e^{x^2}\int_0^\infty a(y)\int_{-1}^1\frac{\exp(-\frac{q(e^{-t/2}y,x,s)}{1-e^{-t}}}{(1-e^{-t})^{\alpha+1}}\psi(x,y,s)\Pi_\alpha(s)dsd\mathfrak{m}_\alpha(y).
\end{align*}
We define
\[
K_t^\alpha(x,y)=\int_{-1}^1\frac{\exp\left(-\frac{q(e^{-t/2}y,x,s)}{1-e^{-t}}\right)}{(1-e^{-t})^{\alpha+1}}\psi(x,y,s)\Pi_\alpha(s)ds,\quad t,x,y\in (0,\infty).
\]
According to \cite[\S 4]{BDQS} we can deduce
\begin{equation}\label{ZZ1}
\sup_{0<t\le m(x)^2}|\partial_yK_t(x,y)|\le C\frac{1}{|x-y|\mathfrak{m}_\alpha(I(x,|x-y|))},\quad x,y\in (0,\infty).
\end{equation}
Since $\mathbb{W}_*^\alpha$ is bounded on $L^2((0,\infty),\gamma_\alpha)$, we deduce that
\begin{align*}
\int_0^\infty |\mathbb{W}_{*,\loc}^\alpha(b)(x)|d\gamma_\alpha(x)&=\int_{2I} |\mathbb{W}_{*,\loc}^\alpha(b)(x)|d\gamma_\alpha(x)\\
&\quad +\int_{(2I)^c} |\mathbb{W}_{*,\loc}^\alpha(b)(x)|d\gamma_\alpha(x)\\
&\le \left(\int_0^\infty  |\mathbb{W}_{*}^\alpha(b)(x)|^2d\gamma_\alpha(x)\right)^{1/2}(\gamma_\alpha(2I))^{1/2}\\
&\quad +\int_{(2I)^c} |\mathbb{W}_{*,\loc}^\alpha(b)(x)|d\gamma_\alpha(x)\\
&\le C\|b\|_{L^2((0,\infty),\gamma_\alpha))}(\gamma_\alpha(2I))^{1/2}\\
&\quad +\int_{(2I)^c} |\mathbb{W}_{*,\loc}^\alpha(b)(x)|d\gamma_\alpha(x)\\
&\le C+\int_{(2I)^c} |\mathbb{W}_{*,\loc}^\alpha(b)(x)|d\gamma_\alpha(x).
\end{align*}
On the other hand, we can write
\begin{align*}
W_{t,\loc}^\alpha(b)(x)&=e^{x^2}\int_0^\infty a(y)K_t^\alpha(x,y)d\mathfrak{m}_\alpha(y)\\
&=e^{x^2}\int_0^\infty a(y)(K_t^\alpha(x,y)-K_t^\alpha(x,c_I))d\mathfrak{m}_\alpha(y),\quad x\in (0,\infty).
\end{align*}
By using \eqref{ZZ1} and \cite[(1.4)]{YY} we get
\begin{align*}
\mathbb{W}_{*,\loc}^\alpha(b)(x)&\le Ce^{x^2}\int_I|a(y)|\frac{|y-c_I|}{|x-y|\mathfrak{m}_\alpha(I(x,|x-y|))}d\mathfrak{m}_\alpha(y)\\
&\le C\frac{e^{x^2}r_I}{\mathfrak{m}_\alpha(I)}\int_I\frac{1}{|x-y|^2x^{2\alpha+1}}d\mathfrak{m}_\alpha(y),\quad x\in (0,\infty).
\end{align*}
It follows that
\[
\int_{(2I)^c}\mathbb{W}_{*,\loc}^\alpha(b)(x)d\gamma_\alpha(x)\le C\frac{r_I}{\mathfrak{m}_\alpha(I)}\int_I\int_{(2I)^c}\frac{dx}{|x-y|^2}d\mathfrak{m}_\alpha(y)\le C.
\]
We conclude that \eqref{L1W*loc} holds with constant independent of the atom $b$.

We now prove that the operator $\mathbb{W}_{*,\glob}^\alpha$ is bounded from $\mathcal{H}^1((0,\infty),\gamma_\alpha)$ into $L^1((0,\infty),\gamma_\alpha)$. Actually we will see that $\mathbb{W}_{*,\glob}^\alpha$ is bounded on $L^1((0,\infty),\gamma_\alpha)$.

We take $\alpha=\frac{k}{2}-1$, with $k\in \mathbb{N}$, $k\ge 2$ and for every $\overline{x}\in \mathbb{R}^k$ we write $x=|\overline{x}|$. For every $\overline{x}, \,\overline{y}\in \mathbb{R}^k$ we have that $|\overline{x}-\overline{y}|^2=q(x,y,\cos(\theta))$, where $\theta$ is the angle between $\overline{x}$ and $\overline{y}$. Then, $(x,y,\cos(\theta))\in L_1$ if and only if $|\overline{x}-\overline{y}|\le \frac{1}{1+x+y}$. We now integrate on $\mathbb{R}^k$ using spherical coordinates by performing the change of variables $s=\cos(\theta)$ to obtain
\begin{align*}
|W_{t,\glob}^\alpha(f)(x)|&\le C\int_{|\overline{x}-\overline{y}|\ge \frac{1}{1+x+y}}\frac{\exp\left(-\frac{|e^{-t/2}\overline{x}-\overline{y}|^2}{1-e^{-t}}\right)}{(1-e^{-t})^{k/2}}|f(y)|d\overline{y}\\
&\le C\int_{|\overline{x}-\overline{y}|\ge \frac{1}{1+x+y}}W_t^{OU}(\overline{x},\overline{y})|f(y)|d\overline{y}, \quad x\in \mathbb{R}^k, t\in (0,\infty).
\end{align*}
Here, for every $\overline{x},\,\overline{y}\in \mathbb{R}^k$ and $t\in (0,\infty)$, $W_t^{OU}(\overline{x},\overline{y})$ represents the integral kernel of the Ornstein-Uhlenbeck semigroup in $\mathbb{R}^k$ (see \cite{HTV}).

We are going to show that  $|\overline{x}-\overline{y}|\ge m(x)$  provided that $|\overline{x}-\overline{y}|\ge \frac{1}{1+x+y}$. In order to do so, we shall consider three cases.
\begin{enumerate}[label=(\alph*)]
    \item Assume that $y\ge 4$ and $y\le 2x$. Then, $x\ge 2$ and $m(x)=\frac{1}{x}$. It follows that 
    \[
    |\overline{x}-\overline{y}|\ge \frac{1}{1+x+y}\ge \frac{1}{1+3x}=\frac{1}{x}\frac{x}{1+3x}\ge C\frac{1}{x}.
    \]
    \item If $y\ge 4$ and $y\ge 2x$, then 
    \[
    |\overline{x}-\overline{y}|\ge y-x \ge \frac{y}{2}\ge 2\ge m(x). 
    \]
    \item Assume that $y\le 4$. We deduce that, for $x\le 1$, 
    \[
    |\overline{x}-\overline{y}|\ge \frac{1}{1+x+y}\ge \frac{1}{5+x}\ge \frac{1}{6}=\frac{1}{6}m(x),
    \]
    while, for $x>1$, 
    \[
    |\overline{x}-\overline{y}|\ge \frac{1}{1+x+y}\ge \frac{1}{5+x}\ge {\frac1x}\frac{x}{5+x}\ge Cm(x).
    \]
\end{enumerate}
Thus our objective is established.

We get 
\[
|W_{t,\glob}^\alpha(f)(x)|\le C\int_{|\overline{x}-\overline{y}|\ge m(x)}W_t^{OU}(\overline{x},\overline{y})|f(y)d\overline{y}, \quad x\in \mathbb{R}^k, t\in (0,\infty).
\]
Let $\omega$ be a measurable function in $(0,\infty)$ such that $\omega(x)\in (0,1)$, $x\in (0,\infty)$. We consider, for every $z\in \mathbb{C}$ with $\Real(z)>-\frac12$, the operator
\[
S_{\omega}^z(f)(x)=e^{-x^2}x^{2z+1}W_{\omega(x),\glob}^z\left(f(y)e^{y^2}y^{-1-2z}\right)(x)\chi_{(0,m(x)^2)}(\omega(x)),\; x\in (0,\infty).
\]
We have that
\begin{enumerate}[label=(\alph*)]
    \item For every $z\in \mathbb{C}$, $\Real(z)>-\frac12$, $S_{\omega}^z$ is a measurable function on $(0,\infty)$ provided that $f$ is a simple function on $(0,\infty)$.
    \item Suppose that $f,g$ 
    are simple functions on $(0,\infty)$. The function 
\[
F(z)=\int_0^\infty S_{\omega}^z(f)(x)g(x)dx,\quad z\in \mathbb{C},\quad \Real(z)>-\tfrac12,
\]
is holomorphic. Furthermore, for every $-1/2<c<d<\infty$, there exists $C>0$ such that $|F(x+iy)|\le C$, $c\le x\le d$ and $y\in \mathbb{R}$.
\end{enumerate}

We have that, for every $\overline{x}\in \mathbb{R}^k$ and $\sigma\in \mathbb{R}$,
\[
\left|S_{\omega}^{\frac{k}{2}-1+i\sigma}(f)(x)\right|\le Ce^{-x^2}x^{k-1}\sup_{0<t<m(x)^2}\int_{|\overline{x}-\overline{y}|\ge  m(x)}W_t^{OU}(\overline{x},\overline{y})|f(y)|y^{-k+1}e^{y^2}d\overline{y}.
\]
According to \cite[Proposition~2.4~(i)]{Po} (see also \cite{GIT}) we obtain
\[
\left\|S_{\omega}^{\frac{k}{2}-1+i\sigma}(f)\right\|_{L^1((0,\infty),dx)}\le C\|f\|_{L^1((0,\infty),dx)}, \quad f\in L^1((0,\infty),dx), \sigma\in \mathbb{R},
\]
where $C>0$ does not depend on $\sigma\in \mathbb{R}$ and the function $\omega$.

By using \cite[Theorem 1]{StInterp} we deduce that, for every $z\in \mathbb{C}$, $\Real(z)\geq0$, there exists $C>0$ such that
\[
\|S_{\omega}^{z}(f)\|_{L^1((0,\infty),dx)}\le C\|f\|_{L^1((0,\infty),dx)}, \quad f\in L^1((0,\infty),dx),
\]
where $C>0$ does not depend on the function $\omega$.

Then, for every $\alpha>0$ there exists $C>0$ independent of $\omega$ such that
\[
\|\mathbb{S}_{\omega}^{\alpha}(f)\|_{L^1((0,\infty),\gamma_\alpha)}\le C\|f\|_{L^1((0,\infty),\gamma_\alpha)}, \quad f\in L^1((0,\infty),\gamma_\alpha),
\]
where
\[
\mathbb{S}_{\omega}^{\alpha}(f)(x)={W}_{\omega(x),\glob}^\alpha(f)(x)\chi_{(0,m(x)^2)}(\omega(x)),\quad x\in (0,\infty).
\]
Let $f\in L^1((0,\infty),\gamma_\alpha)$. We choose $\{t_j\}_{j=1}^\ell\in \mathbb{Q}\cap (0,1)$. For every $x\in (0,\infty)$, we define
\begin{align*}
j(x)=\min\left\{j=1,\ldots,\ell: |W_{t_j,\glob}^\alpha \right.&(f)(x)\chi_{(0,m(x)^2)}(t_j)|\\
&=\left.\max_{k=1,\ldots,\ell}|W_{t_k,\glob}^\alpha(f)(x)\chi_{(0,m(x)^2)}(t_k)|\right\},
\end{align*}
and $\omega(x)=t_{j(x)}$.

We also consider, for every $j=1,\ldots,\ell$, the set
\[
A_j=\left\{x\in (0,\infty):\omega(x)=t_j\right\}.
\]
Note that, for every $j=1,\ldots,\ell$, $x\in A_j$ if and only if
\[
|W_{t_j,\glob}^\alpha(f)(x)
\chi_{(0,m(x)^2)}(t_j)|
=\max_{k=1,\ldots,\ell}|W_{t_k,\glob}^\alpha(f)(x)\chi_{(0,m(x)^2)}(t_k)|,
\]
and, when $j>1$, for every $i=1,\ldots,j-1$,
\[
|W_{t_i,\glob}^\alpha(f)(x)\chi_{(0,m(x)^2)}(t_i)|
<\max_{k=1,\ldots,\ell}|W_{t_k,\glob}^\alpha(f)(x)\chi_{(0,m(x)^2)}(t_k)|.
\]
It follows that $A_j$ is measurable, for every $j=1,\ldots,\ell$.

We can write
\[
\omega=\sum_{j=1}^\ell t_j\chi_{A_j}.
\]
Hence $\omega$ is a measurable function and
\[
|\mathbb{S}_{\omega}^\alpha(f)(x)|=\max_{j=1,\ldots,\ell}|W_{t_j,\glob}^\alpha(f)(x)\chi_{(0,m(x)^2)}(t_j)|, \quad x\in (0,\infty).
\]
Then
\[
\left\|\max_{j=1,\ldots,\ell}|W_{t_j,\glob}^\alpha(f)(x)\chi_{(0,m(x)^2)}(t_j)|\right\|_{L^1((0,\infty),\gamma_\alpha)}\le C\|f\|_{L^1((0,\infty),\gamma_\alpha)}.
\]
Here $C>0$ does not depend on $f$ and $\{t_j\}_{j=1}^\ell$.

We write $\mathbb{Q}\cap (0,1)=\{t_j\}_{j=1}^\infty$. We have that, for any $x\in (0,\infty)$
\[
\sup_{t\in \mathbb{Q}\cap (0,1)}|W_{t,\glob}^\alpha(f)(x)\chi_{(0,m(x)^2)}(t)|=\lim_{k\to\infty}\sup_{j=1,\ldots,k}|W_{t_j,\glob}^\alpha(f)(x)\chi_{(0,m(x)^2)}(t_j)|.
\]
The monotone convergence theorem leads to
\[
\left\|\sup_{t\in \mathbb{Q}\cap (0,1)}|W_{t,\glob}^\alpha(f)(x)\chi_{(0,m(x)^2)}(t)|\right\|_{L^1((0,\infty),\gamma_\alpha)}\le C\|f\|_{L^1((0,\infty),\gamma_\alpha)}.
\]
For every $x\in (0,\infty)$ the function $W_{t,\glob}^\alpha(f)(x)$ is continuous in $t\in (0,\infty)$. Then, for every $x\in (0,\infty)$,
\[
\sup_{\mathbb{Q}\cap (0,\infty)}|W_{t,\glob}^\alpha(f)(x)\chi_{(0,m(x)^2)}(t)|=\sup_{0<t<m(x)^2}|W_{t,\loc}^\alpha(f)(x)|.
\]
We conclude
\[
\|\mathbb{W}_{*,\glob}^\alpha(f)\|_{L^1((0,\infty),\gamma_\alpha)}\le C\|f\|_{L^1((0,\infty),\gamma_\alpha)}.
\]
Thus the proof is finished.


\subsection*{Conflict of interests} The authors declare that there is no conflict of interest.


\bibliographystyle{acm}

\end{document}